\documentclass[a4paper,11pt]{amsart}

\newfont{\cyr}{wncyr10 scaled 1100}

\usepackage[left=2.7cm,right=2.7cm,top=3.5cm,bottom=3cm]{geometry}
\usepackage{amsthm,amssymb,amsmath,amsfonts,mathrsfs,amscd}
\usepackage[latin1]{inputenc}
\usepackage[all]{xy}
\usepackage{latexsym}
\usepackage{longtable}

\theoremstyle{plain}
\newtheorem{theorem}{Theorem}[section]
\newtheorem{corollary}[theorem]{Corollary}
\newtheorem{lemma}[theorem]{Lemma}
\newtheorem{proposition}[theorem]{Proposition}

\theoremstyle{definition}
\newtheorem{definition}[theorem]{Definition}

\theoremstyle{remark}

\newtheorem{remark}[theorem]{Remark}

\newcommand{\J}{\mathbb{J}}

\newcommand{\Q}{\mathbb{Q}}
\newcommand{\Z}{\mathbb{Z}}

\newcommand{\C}{\mathbb{C}}

\newcommand{\PP}{\mathbb{P}}
\newcommand{\Sel}{\operatorname{Sel}}

\newcommand{\Gal}{\operatorname{Gal\,}}

\newcommand{\Div}{\operatorname{Div}}

\newcommand{\ord}{{\operatorname{ord}}}
\newfont{\gotip}{eufb10 at 12pt}

\newcommand{\cO}{{\mathcal O}}

\newcommand{\Pic}{{\mathrm{Pic}}}

\newcommand{\R}{{\mathbb R}}
\newcommand{\M}{{\mathrm{M}}}

\newcommand{\m}{\mathfrak{m}}
\newcommand{\p}{\mathfrak{p}}
\newcommand{\fP}{\mathfrak{P}}

\newcommand{\Gen}{\operatorname{Gen}}

\newcommand{\T}{\mathbb T}

\DeclareMathOperator{\Hom}{Hom} 

\DeclareMathOperator{\Ta}{Ta}

\newcommand{\res}{\mathrm{res}}

\newcommand{\longmono}{\mbox{$\lhook\joinrel\longrightarrow$}}

\newcommand{\smallmat}[4]{\bigl(\begin{smallmatrix}#1&#2\\#3&#4\end{smallmatrix}\bigr)}

\newcommand{\invlim}{\mathop{\varprojlim}\limits}

\newcommand{\D}{\mathbb D}

\include{thebibliography}

\begin{document}

\title[Special values and central derivatives in Hida families]{Vanishing of special values and central derivatives\\in Hida families}

\author{Matteo Longo and Stefano Vigni}

\thanks{}

\begin{abstract} 
The theme of this work is the study of the Nekov\'a\v{r}--Selmer group $\widetilde H^1_f(K,\T^\dagger)$ attached to a twisted Hida family $\T^\dagger$ of Galois representations and a quadratic number field $K$. The results that we obtain have the following shape: if a twisted $L$-function of a suitable modular form in the Hida family has order of vanishing $r\leq1$ at the central critical point then the rank of $\widetilde H^1_f(K,\T^\dagger)$ as a module over a certain local Hida--Hecke algebra is equal to $r$. Under the above assumption, we also show that infinitely many twisted $L$-functions of modular forms in the Hida family have the same order of vanishing at the central critical point. Our theorems extend to more general arithmetic situations results obtained by Howard when $K$ is an imaginary quadratic field and all the primes dividing the tame level of the Hida family split in $K$. 
\end{abstract}

\address{Dipartimento di Matematica Pura e Applicata, Universit\`a di Padova, Via Trieste 63, 35121 Padova, Italy}
\email{mlongo@math.unipd.it}
\address{Department of Mathematics, King's College London, Strand, London WC2R 2LS, United Kingdom}
\email{stefano.vigni@kcl.ac.uk}

\subjclass[2010]{11F11, 11G18}
\keywords{Hida families, quaternionic big Heegner points, $L$-functions, Selmer groups}
\maketitle

\section{Introduction} \label{intro}

Fix an integer $N\geq1$, a prime number $p\nmid 6N\phi(N)$ and an ordinary $p$-stabilized newform 
\[ f=\sum_{n\geq1}a_nq^n\in S_k\bigl(\Gamma_0(Np),\omega^j\bigr) \] 
of weight $k\geq2$ (where $\omega$ is the Teichm\"uller character and $j\equiv k\pmod 2$) whose associated mod $p$ Galois representation is absolutely irreducible and $p$-distinguished (see \cite[p. 297]{LV}). Let $F:=\Q_p(a_n\mid n\geq1)$ be the finite extension of $\Q_p$ generated by the Fourier coefficients of $f$ and let $\mathcal O_F$ denote its ring of integers. Write $\mathcal R$ for the branch of the Hida family where $f$ lives (see \cite[\S 5.3]{LV}); then $\mathcal R$ is a complete local noetherian domain, finite and flat over the Iwasawa algebra $\Lambda:=\mathcal O_F[\![1+p\Z_p]\!]$, and the Hida family we are considering can be viewed as a formal power series $f_\infty\in\mathcal R[\![q]\!]$.  
 
For $\p\in{\rm Spec}(\mathcal R)$ let $\mathcal R_\p$ denote the localization of $\mathcal R$ at $\p$ and set $F_\p:=\mathcal R_\p/\p\mathcal R_\p$. Hida's theory (\cite{Hi1}, \cite{Hida-Galois}, \cite{H}) shows that if $\p\in{\rm Spec}(\mathcal R)$ is an \emph{arithmetic prime} (in the sense of \cite[\S 5.5]{LV}) then $F_\p$ is a finite extension of $F$ and the power series in $F_\p[\![q]\!]$ obtained from $f_\infty$ by composing with the canonical map $\mathcal R\rightarrow F_\p$ is the $q$-expansion of a classical cusp form $f_\p$; in particular, $f=f_{\bar\p}$ for a certain arithmetic prime $\bar\p$ of weight $k$. Furthermore, with $f_\infty$ is associated a $\Gal(\bar\Q/\Q)$-representation $\T$ which is free of rank $2$ over $\mathcal R$ and such that $V_\p:=\T\otimes_\mathcal RF_\p$ is isomorphic to (a Tate twist of) the $\Gal(\bar\Q/\Q)$-representation attached to $f_\p$ by Deligne. The choice of a so-called critical character 
\[ \Theta:\Gal(\bar\Q/\Q)\longrightarrow\Lambda^\times \] 
allows us to uniformily twist $\T$ and obtain a Galois representation $\T^\dagger:=\T\otimes\Theta^{-1}$ which is free of rank $2$ over $\mathcal R$ and such that for every arithmetic prime $\p$ of $\mathcal R$ the representation $V_\p^\dagger:=\T^\dagger\otimes_{\mathcal R}F_\p$ is a self-dual twist of $V_\p$. 

The general theme of this article is the study of the $\mathcal R$-rank of Nekov\'a\v{r}'s Selmer group $\widetilde H_f^1(K,\T^\dagger)$ attached to $\T^\dagger$ and a quadratic number field $K$. 

We first let $K$ be an imaginary quadratic field $K$ of discriminant prime to $Np$ and define the factorization $N=N^+N^-$ by requiring that the primes dividing $N^+$ split in $K$ and the primes dividing $N^-$ are inert in $K$. We assume throughout that $N^-$ is square-free. Moreover, since the case $N^-=1$ was studied in \cite{Ho1}, we also assume that $N^->1$. We say that we are in the \emph{indefinite} (respectively, \emph{definite}) case if $N^-$ is a product of an \emph{even} (respectively, \emph{odd}) number of primes. For every arithmetic prime $\p$ of $\mathcal R$, the choice of $\Theta$ gives an $F_\p^\times$-valued character $\chi_\p$ of $\mathbb A_K^\times$ whose restriction to $\mathbb A_\Q^\times$ is the inverse of the nebentype of $f_\p$ (here $\mathbb A_L$ is the adele ring of a number field $L$). Then the Rankin--Selberg $L$-function $L_K(f_\p,\chi_\p,s)$ admits a self-dual functional equation whose sign controls the order of vanishing of $L_K(f_\p,\chi_\p,s)$ at the critical point $s=1$. We say that the pair $(f_\p,\chi_\p)$ has \emph{analytic rank} $r\geq0$ if the order of vanishing of $L_K(f_\p,\chi_\p,s)$ at $s=1$ is $r$. 

A simplified version of our main results can be stated as follows.

\begin{theorem} \label{thm-intro1} 
Let $K$ be an imaginary quadratic field. If there is an arithmetic prime $\p$ of weight $2$ and non-trivial nebentype such that $(f_\p,\chi_\p)$ has analytic rank one (respectively, zero) then $\widetilde H^1_f(K,\T^\dagger)$ is an $\mathcal R$-module of rank one (respectively, an $\mathcal R$-torsion module).
\end{theorem}

Proposition \ref{prop5.2} (respectively, Proposition \ref{prop4.2}) implies that if the pair $(f_\p,\chi_\p)$ has analytic rank one (respectively, zero) then we are in the indefinite (respectively, definite) case. Under this assumption, we also show that $(f_{\p'},\chi_{\p'})$ has analytic rank one (respectively, zero) for all but finitely many arithmetic primes $\p'$ of weight $2$; see Corollary \ref{equivalence} (respectively, Corollary \ref{coro-def}) for details. These results represent weight $2$ analogues (over $K$) of the conjectures on the generic analytic rank of the forms $f_\p$ formulated by Greenberg in \cite{greenberg}, and the reader is suggested to compare them with the statements in \cite[\S 9.4]{LV} (in particular, \cite[Conjecture 9.13]{LV}), where conjectures for forms $f_\p$ of arbitrary weight are proposed. Theorem \ref{thm-intro1} is a combination of Theorems \ref{teo-indef} and \ref{teo-def}, where more general formulations (involving the notion of \emph{generic} arithmetic primes of weight $2$, cf. Definition \ref{def-gen}) can be found. In the case of rank one, the corresponding statement in the split (i.e., $N^-=1$) setting can be proved using the results obtained by Howard in \cite[\S 3]{Ho2}.    

The proof of Theorem \ref{thm-intro1} given in Sections \ref{sec4} and \ref{sec5} relies crucially on the properties of the \emph{big Heegner points} whose construction is described in \cite{LV}, where we generalized previous work of Howard on the variation of Heegner points in Hida families (\cite{Ho1}). With a slight abuse of terminology (cf. Definition \ref{big-def}), for the purposes of this introduction we can say that big Heegner points are certain elements $\mathfrak Z\in H^1(K,\T^\dagger)$ (indefinite case) and $\mathfrak Z\in \J$ (definite case) defined in terms of distributions of suitable Heegner (or Gross--Heegner) points on Shimura curves, where $\J$ is an $\mathcal R$-module obtained from the inverse limit of the Picard groups attached to a tower of definite Shimura curves with increasing $p$-level structures. Since the Galois representation $\T^\dagger$ can be introduced using an analogous tower of indefinite (compact) Shimura curves, the construction of the elements $\mathfrak Z$ proceeds along similar lines in both cases. These constructions are reviewed in Section \ref{sec3}. It turns out that the image of $\mathfrak Z$ in the $F_\p$-vector space $H^1(K,V_\p^\dagger)$ (indefinite case) or $\J\otimes_\mathcal RF_\p$ (definite case) as $\p$ varies in the set of (weight $2$) arithmetic primes controls the rank of $\widetilde H^1_f(K,\T^\dagger)$. In fact, in light of Propositions \ref{prop4.2} and \ref{prop5.2}, Theorem \ref{thm-intro1} tells us that knowledge of the local non-vanishing of the big Heegner point $\mathfrak Z$ at a single arithmetic prime $\p$ of weight $2$ and non-trivial nebentype allows us to predict the rank over $\mathcal R$ of the big Selmer group $\widetilde H^1_f(K,\T^\dagger)$.

Finally, by suitably splitting $\widetilde H_f^1(K,V_\p^\dagger)$ and considering twists of Hida families, one can obtain results in the same vein as Theorem \ref{thm-intro1} for general quadratic fields $K$ (not only imaginary ones) from the results (over $\Q$) proved in \cite[Section 4]{Ho2}. This strategy does not make use of the big Heegner points introduced in \cite{LV}, and the input has to be an arithmetic prime of \emph{trivial} character. We provide details in Section \ref{complementary-subsec} (see Theorems \ref{complementary-thm} and \ref{complementary-thm-bis}).
\vskip 2mm
\noindent\emph{Convention.} Throughout the paper we fix algebraic closures $\bar\Q$, $\bar\Q_p$ and field embeddings $\bar\Q\hookrightarrow\bar\Q_p$, $\bar\Q_p\hookrightarrow\C$.

\section{Hida families of modular forms} \label{hida-sec}

Define $\Gamma:=1+p\Z_p$ (so $\Z_p^\times=\boldsymbol\mu_{p-1}\times\Gamma$ where $\boldsymbol\mu_{p-1}$ is the group of $(p-1)$-st roots of unity) and $\Lambda:=\mathcal O_F[\![\Gamma]\!]$. Let $\mathcal R$ denote the integral closure of $\Lambda$ in the primitive component $\mathcal K$ to which $f$ belongs (see \cite[\S 5.3]{LV} for precise definitions). Then $\mathcal R$ is a complete local noetherian domain, finitely generated and flat as a $\Lambda$-module, and if $\mathfrak h_\infty^\ord$ is Hida's ordinary Hecke algebra over $\mathcal O_F$ of tame level $\Gamma_0(N)$, whose construction is recalled in \cite[\S 5.1]{LV}, then there is a canonical map 
\[ \phi_\infty: \mathfrak h_\infty^\ord\longrightarrow\mathcal R \]
which we shall sometimes call the \emph{Hida family} of $f$. For every arithmetic prime $\p$ of $\mathcal R$ (see \cite[\S 5.5]{LV} for the definition) set $F_\p:=\mathcal R_\p/\p\mathcal R_\p$. Then $F_\p$ is a finite field extension of $F$ and we fix an $F$-algebra embedding $F_\p\hookrightarrow\bar\Q_p$, so that from here on we shall view $F_\p$ as a subfield of $\bar\Q_p$. Now consider the composition 
\begin{equation} \label{char-p} 
\phi_\p:\mathfrak h_\infty^\ord\overset{\phi_\infty}\longrightarrow\mathcal R\overset{\pi_\p}\longrightarrow F_\p, 
\end{equation}
where $\pi_\p$ is the map induced by the canonical projection $\mathcal R_\p\twoheadrightarrow F_\p$. For every integer $m\geq1$ define $\Phi_m:=\Gamma_0(N)\cap\Gamma_1(p^m)$. By duality, with $\phi_\p$ is associated a modular form 
\[ f_\p=\sum_{n\geq1}a_n(\p)q^n\in S_{k_\p}\bigl(\Phi_{m_\p},\psi_\p\omega^{k+j-k_\p},F_\p\bigr), \] 
of suitable weight $k_\p$, level $\Phi_{m_\p}$ and finite order character $\psi_\p:\Gamma\rightarrow F_\p^\times$, such that $a_n(\p)=\phi_\p(T_n)\in F_\p$ for all $n\geq1$ (here $T_n\in\mathfrak h_\infty^\ord$ is the $n$-th Hecke operator). We remark that if $t_\p$ is the smallest positive integer such that $\psi_\p$ factors through $(1+p\Z_p)/(1+p^{t_\p}\Z_p)$ then $m_\p=\max\{1,t_\p\}$. See \cite[\S 5.5]{LV} for details. 

Let $\Q(\boldsymbol\mu_{p^\infty})$ be the $p$-adic cyclotomic extension of $\Q$ and factor the cyclotomic character 
\[ \epsilon_{\rm cyc}:\Gal(\Q(\boldsymbol{\mu}_{p^\infty})/\Q)\overset\simeq\longrightarrow\Z_p^\times \] 
as a product $\epsilon_{\rm cyc}=\epsilon_{\rm tame}\epsilon_{\rm wild}$ with $\epsilon_{\rm tame}$ (respectively, $\epsilon_{\rm wild}$) taking values in $\boldsymbol\mu_{p-1}$ (respectively, $\Gamma$). Write $\gamma\mapsto[\gamma]$ for the inclusion $\Gamma\hookrightarrow\Lambda^\times$ of group-like elements and define 
\[ \Theta:=\epsilon_{\rm tame}^{(k+j-2)/2}\bigl[\epsilon_{\rm wild}^{1/2}\bigr]:\Gal(\Q(\boldsymbol{\mu}_{p^\infty})/\Q)\longrightarrow\Lambda^\times \] 
where $\epsilon_{\rm wild}^{1/2}$ is the unique square root of $\epsilon_{\rm wild}$ with values in $\Gamma$. In the obvious way, we shall also view $\Theta$ as $\mathcal R^\times$-valued. We associate with $\Theta$ the character $\theta$ and, for every arithmetic prime $\p$, the characters $\Theta_\p$ and $\theta_\p$ given by  
\[ \theta:\Z_p^\times\xrightarrow{\epsilon_{\rm cyc}^{-1}}\Gal(\Q(\boldsymbol\mu_{p^\infty})/\Q)\overset\Theta\longrightarrow\mathcal R^\times, \]
\[ \Theta_\p:\Gal(\Q(\boldsymbol\mu_{p^\infty})/\Q)\overset\Theta\longrightarrow\mathcal R^\times\overset{\pi_\p}\longrightarrow F_\p^\times, \]
\[ \theta_\p:\Z_p^\times\xrightarrow{\epsilon_{\rm cyc}^{-1}}\Gal(\Q(\boldsymbol\mu_{p^\infty})/\Q)\overset{\Theta_\p}\longrightarrow F_\p^\times. \]
According to convenience, the characters $\Theta$ and $\Theta_\p$ will also be viewed as defined on $\Gal(\bar\Q/\Q)$ via the canonical projection. If we write $[\cdot]_\p$ for the composition of $[\cdot]$ with $\mathcal R^\times\rightarrow F^\times_\p$ then $({\theta_\p|}_\Gamma)^2=[\cdot]_\p$, and for every arithmetic prime $\p$ of weight $2$ the modular form $f_\p^\dagger:=f_\p\otimes\theta_\p^{-1}$ has trivial nebentype (see \cite[p. 806]{Ho2}). 

As in \cite[Definition 2]{Ho2}, we give
 
\begin{definition} \label{def-gen}
An arithmetic prime $\p$ of weight $2$ is \emph{generic} for $\theta$ if one of the following (mutually exclusive) conditions is satisfied: 
\begin{enumerate}
\item $f_\p$ is the $p$-stabilization (in the sense of \cite[Definition 1]{Ho2}) of a newform in $S_2(\Gamma_0(N),F_\p)$ and $\theta_\p$ is trivial; 
\item $f_\p$ is a newform in $S_2(\Gamma_0(Np),F_\p)$ and $\theta_\p=\omega^{(p-1)/2}$; 
\item $f_\p$ has non-trivial nebentype. 
\end{enumerate}
\end{definition}

Let $\Gen_2(\theta)$ denote the set of all weight $2$ arithmetic primes of $\mathcal R$ which are generic for $\theta$. Observe that, since only a finite number of weight $2$ arithmetic primes do not satisfy condition (3) in Definition \ref{def-gen}, the set $\Gen_2(\theta)$ contains all but finitely many weight $2$ arithmetic primes. 

Before concluding this subsection, let us fix some more notations. If $E$ is a number field denote $\mathbb A_E$ its ring of adeles and write ${\rm art}_E:\mathbb A_E^\times\rightarrow\Gal(E^{\rm ab}/E)$ for the Artin reciprocity map. Furthermore, let $\mathfrak h_m$ be the $\mathcal O_F$-Hecke algebra acting on the $\C$-vector space $S_2(\Phi_m,\C)$ and write $\mathfrak h_m^\ord$ for the ordinary $\mathcal O_F$-subalgebra of $\mathfrak h_m$, which is defined as the product of the local summands of $\mathfrak h_m$ where $U_p$ is invertible (therefore $\mathfrak h_\infty^\ord=\varprojlim\mathfrak h_m^\ord$, cf. \cite[\S 5.1]{LV}). If $M$ is an $\mathfrak h_m$-module set $M^\ord:=M\otimes_{\mathfrak h_m}\mathfrak h_m^\ord$. Finally, define the $\Theta^{-1}$-twist of a $\Lambda$-module $M$ endowed with an action of $\Gal(\bar\Q/\Q)$ as follows: let $\Lambda^\dagger$ denote $\Lambda$ viewed as a module over itself but with $\Gal(\bar\Q/\Q)$-action given by $\Theta^{-1}$ and set $M^\dagger:=M\otimes_\Lambda\Lambda^\dagger$. 

\section{Review of big Heegner points} \label{sec3}

Let $B$ denote the quaternion algebra over $\Q$ of discriminant $N^-$. Fix an isomorphism of $\Q_p$-algebras $i_p:B_p:=B\otimes_\Q\Q_p\simeq M_2(\Q_p)$ and choose for every integer $t\geq0$ an Eichler order $R_t$ of level $p^tN^+$ such that $i_p(R_t\otimes_\Z\Z_p)$ is equal to the order of $\M_2(\Q_p)$ consisting of the matrices in $\M_2(\Z_p)$ which are upper triangular matrices modulo $p^t$ and $R_{t+1}\subset R_t$ for all $t$. Let $\widehat B$ and $\widehat R_t$ denote the finite adelizations of $B$ and $R_t$ and define $U_t$ to be the compact open subgroup of $\widehat R_t^\times$ obtained by replacing the $p$-component of $\widehat R_t^\times$ with the subgroup of $(R_t\otimes_\Z\Z_p)^\times$ consisting of those elements $\gamma$ such that $i_p(\gamma)\equiv\smallmat 1b0d\pmod{p^t}$. 

Fix an integer $t\geq0$ and define the set 
\begin{equation} \label{heeg-eq} 
\widetilde X_t^{(K)}=U_t\big\backslash\bigl(\widehat B^\times\times\Hom_\Q(K,B)\bigr)\big/B^\times, 
\end{equation} 
where $\Hom_\Q$ denotes homomorphisms of $\Q$-algebras. See \cite[\S2.1]{LV} for details. An element $[(b,\phi)]\in\widetilde X_t^{(K)}$ is said to be a \emph{Heegner point of conductor $p^t$} if the following two conditions are satisfied.
\begin{enumerate}
\item The map $\phi$ is an optimal embedding of the order $\mathcal O_{p^t}$ of $K$ of conductor $p^t$ into the Eichler order $b^{-1}\widehat R_tb\cap B$ of $B$.
\item Let $x_p$ denote the $p$-component of an adele $x$, define $U_{t,p}$ to be the $p$-component of $U_t$, set $\mathcal O_{p^t,p}:=\mathcal O_{p^t}\otimes_\Z\Z_p$, $\mathcal O_{K,p}:=\mathcal O_K\otimes_\Z\Z_p$ and let $\phi_p$ denote the morphism obtained from $\phi$ by extending scalars to $\Z_p$. Then  
\[ \phi_p^{-1}\bigl(\phi_p(\mathcal O_{p^t,p}^\times)\cap b_p^{-1}U_{t,p}^\times b_p\bigr)=\mathcal O_{p^t,p}^\times\cap(1+p^t\mathcal O_{K,p})^\times. \]
\end{enumerate}
For every $t\geq0$ let $H_t$ be the ring class field of $K$ of conductor $p^t$ and set $L_t:=H_t(\boldsymbol\mu_{p^t})$.

\subsection{Indefinite Shimura curves}

For every $t\geq0$ we consider the indefinite Shimura curve $\widetilde X_t$ over $\Q$ whose complex points are described as the compact Riemann surface 
\[ \widetilde X_t(\C):=U_t\big\backslash\bigl(\widehat B^\times\times(\C-\R)\bigr)\big/B^\times. \]
The set $\widetilde X_t^{(K)}$ can be identified with a subset of $\widetilde X_t(K^{\rm ab})$ where $K^{\rm ab}$ is the maximal abelian extension of $K$. For every number field $E$ let 
\[ \mathcal G_E:=\Gal(\bar\Q/E). \] 
Then $\mathcal G_E$ acts naturally on the geometric points $\widetilde X_t(\bar\Q)$ of $\widetilde X_t$. 

\subsection{Definite Shimura curves} \label{definite-shimura-subsec}

For every $t\geq0$ we consider the definite Shimura curve 
\[ \widetilde X_t:=U_t\backslash(\widehat B^\times\times\PP)/B^\times \]
where $\PP$ is the genus $0$ curve over $\Q$ defined by setting  
\[ \PP(A):=\bigl\{x\in B\otimes_\Q A\mid\text{$x\neq0$, ${\rm Norm}(x)={\rm Trace}(x)=0$}\bigr\}\big/A^\times \]
for every $\Q$-algebra $A$ (here ${\rm Norm}$ and ${\rm Trace}$ are the norm and trace maps of $B$, respectively). There is a canonical identification between the set $\widetilde X_t^{(K)}$ in \eqref{heeg-eq} and the $K$-rational points $\widetilde X_t(K)$ of $\widetilde X_t$. 

If $E/K$ is an abelian extension then we define 
\[ \mathcal G_E:=\Gal(K^{\rm ab}/E); \]
in particular, $\mathcal G_K$ is the abelianization of the absolute Galois group of $K$. We explicitly remark that this group is 
different from the group denoted by the same symbol in the indefinite case. The set $\widetilde X_t^{(K)}$ is equipped with a canonical action of $\mathcal G_K\simeq\widehat K^\times/K^\times$ given by 
\[ [(b,\phi)]^\sigma:=\bigl[(b\hat\phi(\mathfrak a),\phi)\bigr], \] 
where $\hat\phi:\widehat K\hookrightarrow\widehat B$ is the adelization of $\phi$ and $\mathfrak a\in\widehat K^\times$ satisfies ${\rm art}_K(\mathfrak a)=\sigma$.

\subsection{Constructions of points}

In this and the next subsection our considerations apply both to the definite case and to the indefinite case. Recall that in the former case the symbol $\mathcal G_E$ stands for the Galois group $\Gal(K^{\rm ab}/E)$ of an abelian extension $E$ of $K$, while in the latter we use $\mathcal G_E$ for the absolute Galois group of a number field $E$. Given a field extension $E/F$ (with $E$ and $F$ abelian over $K$ in the definite case) and a continuous $\mathcal G_F$-module $M$, for every integer $i\geq0$ we let
\begin{equation} \label{res-cores-eq}
\res_{E/F}:H^i(\mathcal G_F,M)\longrightarrow H^i(\mathcal G_E,M),\qquad{\rm cor}_{E/F}:H^i(\mathcal G_E,M)\longrightarrow H^i(\mathcal G_F,M)
\end{equation}
be the usual restriction and corestriction maps in Galois cohomology. 

Denoting $\alpha_t:\widetilde X_t\rightarrow\widetilde X_{t-1}$ the canonical projection, \cite[Theorem 1.1]{LV} shows the existence of a family of points $\widetilde P_t\in\widetilde X_t^{(K)}$ which are fixed by the subgroup $\mathcal G_{L_t}$ of $\mathcal G_K$ and satisfy 
\[ U_p\bigl(\widetilde P_{t-1}\bigr)=\alpha_t\Big({\rm tr}_{L_t/L_{t-1}}\bigl(\widetilde P_t\bigr)\Big) \]
(here, for a finite Galois extension $E/F$, the symbol ${\rm tr}_{E/F}$ stands for the usual trace map). Via the Jacquet--Langlands correspondence, the divisor group $\Div\bigl(\widetilde X_t^{(K)}\bigr)\otimes_\Z\mathcal O_F$ is equipped with a standard action of the $\mathcal O_F$-Hecke algebra $\mathfrak h_t$ (see \cite[\S 6.3]{LV}). Define 
\[ \D_t:=\Big(\!\Div\bigl(\widetilde X_t^{(K)}\bigr)\otimes\mathcal O_F\!\Big)^{\!\ord} \]
(notice that here and below we are regarding $\Theta$ as a character of $\mathcal G_K$ via restriction). Taking the inverse limit, we may define 
\[ \D_\infty:=\invlim_t\D_t, \] 
which is naturally an $\mathfrak h_\infty^\ord$-module. Finally, set 
\[ \D:=\D_\infty\otimes_{\mathfrak h_{\infty}^{\rm ord}}\mathcal R. \] 
The canonical structure of $\Lambda$-module on these groups makes it possible to define the twisted modules $\D_t^\dagger$, $\D_\infty^\dagger$ and $\D^\dagger$. 

Let $H=H_0$ be the Hilbert class field of $K$, write 
\[ \PP_t\in H^0(\mathcal G_{H_{t}},\D_t^\dagger) \] 
for the image of $\widetilde P_t$ in $\D_t$ (cf. \cite[\S 7.1]{LV}) and define 
\[ \mathcal P_t:={\rm cor}_{H_{t}/H}(\PP_t)\in H^0(\mathcal G_H,\D_t^\dagger). \]

Following \cite[Definition 7.3]{LV}, we give
\begin{definition} \label{big-def}
The \emph{big Heegner point of conductor $1$} is the element 
\[ \mathcal P_H:=\invlim_tU_p^{-t}(\mathcal P_t)\in H^0(\mathcal G_H,\D^\dagger)\]
obtained by taking the inverse limit of the compatible sequence $\bigl(U_p^{-t}(\mathcal P_t)\bigr)_{t\geq1}$ and then taking the image via the map $\D_\infty^\dagger\rightarrow\D^\dagger$.  
\end{definition}
See \cite[Corollary 7.2]{LV} for a proof that the inverse limit considered in Definition \ref{big-def} makes sense. Finally, set
\begin{equation}\label{def-P}
\mathcal P=\mathcal P_K:={\rm cor}_{H/K}(\mathcal P_H)\in H^0(\mathcal G_K,\D^\dagger). \end{equation}

\subsection{Weight $2$ arithmetic primes}\label{sec3.4}
We introduce some notations in the case of arithmetic primes of weight 2. 
So fix an arithmetic prime $\p$ of $\mathcal R$ of weight $2$, level $\Phi_{m_\p}$ and character $\psi_\p$. In order to simplify notations, in this subsection put $m:=m_\p$ and $\psi:=\psi_\p$. The action of $\mathfrak h_\infty^\ord$ on $D_m^\ord$ factors through $\mathfrak h_m^\ord$, which, by the Jacquet--Langlands correspondence, can be identified with a quotient of the ordinary $\mathcal O_F$-Hecke algebra of level $\Phi_m$.  
Define an $F_\p^\times$-valued character of $\mathbb A_K^\times$ by 
\[ \chi_\p(x):=\Theta_\p\bigl({\rm art}_\Q({\rm N}_{K/\Q}(x))\bigr) \] 
and denote $\chi_{0,\p}$ the restriction of $\chi_\p$ to $\mathbb A_\Q^\times$. Then the nebentype $\psi_\p\omega^{k+j-k_\p}$ of $f_\p$ is equal to $\chi_{0,\p}^{-1}$ (see \cite[\S 3]{Ho2} for a proof). In particular, since 
\[ (\Z/p\Z)^\times\times(1+p\Z_p)/(1+p^{t_\p}\Z_p)\simeq(\Z/p^{t_\p}\Z)^\times, \] 
the conductor of $\chi_{0,\p}^{-1}$ is $p^{t_\p}$ unless $\chi_{0,\p}$ is trivial. For $\sigma\in\Gal(K^{\rm ab}/K)$ let $x_\sigma\in\mathbb A^\times_K$ be such that ${\rm art}_K(x_\sigma)=\sigma$ and define $\chi_\p(\sigma):=\chi_\p(x_\sigma)$. 

Since $\Theta_\p$ factors through $\mathcal G_{\Q(\boldsymbol\mu_{p^m})}$, we obtain the equality 
\begin{equation}\label{eq'}
H^0(\mathcal G_{L_m},\D^\dagger_m)=
H^0(\mathcal G_{L_m},\D_m). \end{equation}
Using \eqref{eq'}, the point $\PP_m\in H^0(\mathcal G_{H_m},\D_m^\dagger)$ gives rise, by restriction, 
to a point $\res_{L_m/H_m}(\PP_m)$ in $H^0(\mathcal G_{L_m},\D_m)$. 
Since $\mathcal G_{\Q(\boldsymbol\mu_{p^m})}$ is contained in the kernel of $\Theta_\p$, the map $\chi_\p$ can be viewed as a character of $\Gal(L_m/K)$, so we may form the sum
\begin{equation} \label{pp-eq}
\PP_{m}^{\chi_\p}:=\sum_{\sigma\in\Gal(L_m/K)}\res_{L_m/H_m}(\PP_m)^\sigma\otimes\chi_\p^{-1}(\sigma)\in H^0(\mathcal G_{L_m},\D_m\otimes_{\mathcal O_F}F_\p).
\end{equation}
Via \eqref{eq'}, we may also view $\PP_m^{\chi_\p}$ as an element of $H^0(\mathcal G_{L_m},\D_m^\dagger\otimes_{\mathcal O_F}F_\p)$. Moreover, for every $\tau\in\mathcal G_K$ we have
\[ (\PP_m^{\chi_\p})^\tau=\sum_{\sigma\in\Gal(L_m/K)}\PP_m^{\sigma\tau|_{L_m}}\otimes\Theta_\p\bigl({\rm art}_\Q({\rm N}_{K/\Q}(x_\tau))\bigr)\chi_\p^{-1}(\sigma\tau|_{L_m})=\Theta(\tau)\PP_m^{\chi_\p}, \]
so $\PP_m^{\chi_\p}\in H^0(\mathcal G_K,\D_m^\dagger\otimes_{\mathcal O_F}F_\p)$. 

We want to explicitly relate $\PP_m^{\chi_\p}$ and $\mathcal P_m$. In order to do this, we note that, by definition, there is an equality
\[ \PP_m^{\chi_\p}={\rm cor}_{L_m/K}\circ\res_{L_m/H_m}(\PP_m), \] 
where the restriction and corestriction maps are as in \eqref{res-cores-eq} with $E=L_m$, $F=H_m$ and $M=\D_m^\dagger\otimes_{\mathcal O_F}F_\p$. Therefore we get
\begin{equation} \label{2}
\begin{split}
\PP_m^{\chi_\p}&=({\rm cor}_{H/K}\circ{\rm cor}_{H_m/H}\circ{\rm cor}_{L_m/H_m}\circ\res_{L_m/H_m})(\PP_m)\\
&=[L_m:H_m]{\rm cor}_{H/K}(\mathcal P_m).
\end{split}
\end{equation}
In the following we shall also make use of the module 
\[ \D_\p:=\D\otimes_{\mathcal R} F_\p, \] 
its twist $\D_\p^\dagger=\D^\dagger\otimes_{\mathcal R}F_\p$ and  the image of $\mathcal P$ in $H^0(\mathcal G_K,\D^\dagger_\p)$ via the canonical map $\D\rightarrow \D_\p$, which we denote $\mathcal P_\p$. 

\begin{remark}
By extending Hida's arguments in the proof of \cite[Theorem 12.1]{H} to the case of divisors, it seems possible to show that the canonical map $\D\rightarrow \D_\p$ factors as
\begin{equation} \label{Factor}
\D\longrightarrow \D_m\otimes_{\mathcal O_F}F_\p\longrightarrow \D_\p.
\end{equation}
Then one would obtain from the above discussion the formula $\PP_\p=a_p(\p)^m[L_m:H_m]\mathcal P_\p$ in 
$H^0(\mathcal G_K,\D_\p^\dagger)$. Instead of deriving such an equality at this stage, we will obtain two related formulas (see Propositions \ref{lemma1} and \ref{lemma2}) in specific situations where factorizations analogous to \eqref{Factor} are available thanks to known cases of  Hida's control theorem.   
\end{remark}

\section{Vanishing of central derivatives} \label{sec4}

Let the quaternion algebra $B$ be \emph{indefinite} and write $\p$ for a weight $2$ arithmetic prime of $\mathcal R$ of level $\Phi_{m_\p}$ and character $\psi_\p$. Once $\p$ has been fixed, set $m:=m_\p$ and $\psi:=\psi_\p$. 

\subsection{Tate modules}

For every integer $t\geq0$ let $\widetilde J_t$ be the Jacobian variety of $\widetilde X_t$. Define 
\[ \T_t:=\Big(\!\Ta_p\bigl(\widetilde J_t\bigr)\otimes_{\Z_p}\mathcal O_F\!\Big)^{\!\ord} \]
where, as usual, the upper index $\ord$ denotes the ordinary submodule, then form the inverse limit
\[ \T_\infty:=\invlim_t\T_t. \] 
We may also define the twists $\T_t^\dagger$ and $\T_\infty^\dagger$. By \cite[Corollary 6.5]{LV}, there is an isomorphism
\begin{equation} \label{T}
\T_\infty\otimes_{\mathfrak h_\infty^\ord}\mathcal R\simeq\T 
\end{equation}
of Galois representations. We fix once and for all such an isomorphism, and freely use the notation $\T$ for the Galois module appearing in the left hand side of \eqref{T}. Now set 
\[ V_\p:=\T\otimes_\mathcal RF_\p. \]
Then $V_\p^\dagger=\T^\dagger\otimes_{\mathcal R}F_\p$ is (a twist of) the self-dual representation attached to $f_\p^\dagger$ by Deligne. Observe that the canonical map $\T\rightarrow V_\p$ can be factored as
\begin{equation} \label{factor}
\T\longrightarrow \T_m\otimes_{\mathcal O_F}F_\p\xrightarrow{\pi_{m,\p}}V_\p.
\end{equation} 
This follows from Hida's control theorem \cite[Theorem 12.1]{H} (or, rather, from its extension to the current setting given in \cite[Proposition 2.17]{LV-preprint}). Following \cite[p.101]{Ho1}, for every $t\geq0$ one defines a twisted Kummer map
\[ \delta_t:H^0(K,\D^\dagger_t)\longrightarrow H^1(K,\T_t^\dagger). \]
Taking the limit as $m$ varies and using the canonical map $\T_\infty^\dagger\rightarrow\T^\dagger$, we get a map
\[ \delta_\infty:H^0(K,\D^\dagger)\longrightarrow H^1(K,\T^\dagger). \]
Recall the element $\mathcal P$ introduced in \eqref{def-P} and define 
\[ \mathfrak Z:=\delta_\infty(\mathcal P)\in H^1(K,\T^\dagger).\] 
Write $\mathfrak Z_\p$ for the image of $\mathfrak Z$ in $H^1(K,V_\p^\dagger)$. The map $\delta_\infty$ gives rise to a map
\[ \delta_\p:H^0(K,\D^\dagger_\p)\longrightarrow H^1(K,V_\p^\dagger), \]
and the square
\[ \xymatrix{H^0(K,\D^\dagger)\ar[r]^-{\delta_\infty}\ar[d]& H^1(K,\T^\dagger)\ar[d]\\H^0(K,\D^\dagger_\p)\ar[r]^-{\delta_\p}& H^1(K,V_\p^\dagger)} \]
is commutative, so that $\mathfrak Z_\p$ is also equal to $\delta_\p(\mathcal P_\p)$ (in the notations of \S \ref{sec3.4}). We may also consider the class $\delta_m(\PP_m^{\chi_\p})\in H^1(K,\T_m^\dagger\otimes_{\mathcal O_F}F_\p)$, where $\PP_m^{\chi_\p}\in H^0(K,\D^\dagger_m\otimes_{\mathcal O_F}F_\p)$ is the element introduced in \eqref{pp-eq} and $\delta_m$ is extended by $F_p$-linearity, and define $\mathfrak X_\p$ to be the image of $\delta_m(\PP_m^{\chi_\p})$ in $H^1(K,V_\p^\dagger)$ via the map $\pi_{m,\p}$ appearing in \eqref{factor}. 

\begin{proposition} \label{lemma1}
The formula $\mathfrak X_\p=a_p(\p)^m[L_m:H_{m}]\mathfrak Z_\p$ holds in $H^0(\mathcal G_K,\T_\p^\dagger)$.
\end{proposition}

\begin{proof} We just need to review the above constructions. Thanks to factorization \eqref{factor}, the class $\delta_\p(\mathcal P_\p)=\mathfrak Z_\p$ is equal to the corestriction from $H$ to $K$ of the image of $U_p^{-m}\delta_m(\mathcal P_m)\in H^0(\mathcal G_H,\T_m^\dagger)$ in $H^0(\mathcal G_K,V_\p^\dagger)$ under the map induced by the map labeled $\pi_{m,\p}$ in \eqref{factor}. Since the action of $\mathfrak h_\infty^\ord$ on $V_\p^\dagger$ is via the character $\phi_\p$ defined in \eqref{char-p}, one has   
\[ \delta_\p(\mathcal P_\p)=\pi_{m,\p}\Big({\rm cor}_{H/K}\bigl(a_p(\p)^{-m}\delta_m(\mathcal P_m)\bigr)\!\Big). \]
Finally, using \eqref{2} we find that
\[ \delta_\p(\mathcal P_\p)=a_p(\p)^{-m}[L_m:H_m]^{-1}\pi_{m,\p}\bigl(\delta_m(\mathbb P_m^{\chi_\p})\bigr)=a_p(\p)^{-m}[L_m:H_m]^{-1}\mathfrak X_\p, \]
and the result follows. \end{proof} 

\begin{corollary} \label{prop4.1}
$\mathfrak Z_\p\not=0\Longleftrightarrow\mathfrak X_\p\not=0$. 
\end{corollary}

\begin{proof} Immediate from Proposition \ref{lemma1}. \end{proof}

\subsection{Non-vanishing results}\label{sec4.2}

Using the fixed embedding $\bar\Q_p\hookrightarrow\C$, we form the Rankin--Selberg convolution $L_K(f_\p,\chi_\p,s)$ as in \cite[\S 1]{Ho3}; note that if $\boldsymbol1$ is the trivial character then $L_K(f_\p,\boldsymbol1,s)=L_K(f_\p,s)$. For the definition of the $L$-function $L_K(f_\p,s)$ of $f_\p$ over $K$ and of its twists the reader is referred also to \cite[\S 2.1]{BDIS}. The results obtained in this subsection are the counterpart in an arbitrary indefinite quaternionic setting of those in \cite[Propositions 3 and 4]{Ho2}, where the case of the split algebra $\M_2(\Q)$ is considered. Although many arguments follow those in \cite{Ho2} closely, for the convenience of the reader we give detailed proofs. 

\begin{proposition} \label{prop4.2}
Suppose $\p\in\Gen_2(\theta)$ satisfies either (2) or (3) in Definition \ref{def-gen}. 
\begin{itemize}
\item[(a)] $L_K(f_\p,\chi_\p,s)$ vanishes to odd order at $s=1$.
\item[(b)] $L'_K(f_\p,\chi_\p,1)\neq0\Longleftrightarrow\mathfrak Z_\p\neq0$. 
\end{itemize} 
\end{proposition}

\begin{proof} Let $\mathfrak h$ denote the algebra generated over $\Z$ by the Hecke operators $T_n$ with $(n,Np)=1$ and the diamond operators $\langle d\rangle$ with $(d,Np)=1$ acting on the $\C$-vector space of cusp forms in $S_2(\Phi_m,\C)$. Write $e_\p$ for the idempotent in $\mathfrak h\otimes_\Z F_\p$ which projects onto the maximal summand on which every $T_n$ acts as $a_\p(n)$ and $\langle\cdot\rangle$ acts as $\chi_{0,\p}^{-1}$. 

We apply the results in \cite{Ho3} with $\mathfrak s=p^m$, $\mathfrak c=p^m$ and $\mathfrak m=N$. As recalled in \cite[p. 829]{Ho3} (cf. also \cite[eqs. (2.5), (2.6), (2.8)]{Ho3}), the sign of the functional equation for $L_K(f_\p,\chi_\p,s)$ is $-1$, and this shows part (a).
 
Let us prove (b). As in \cite[\S 5.1]{Ho3}, denote $\rm Hg$ the Hodge embedding defined in \cite[\S 6.2]{Zh}. We introduce the divisor 
\[ P:=\sum_{\sigma\in\Gal(L_m/K)}\PP_m^\sigma\otimes\chi_\p^{-1}(\sigma)\in\Div\bigl(\widetilde X_m^{(K)}\bigr)\otimes_\Z F_\p, \]
whose image in $\D_\p$ under the canonical maps is $\PP_\p$. After fixing an appropriate embedding $K\hookrightarrow B$, the image ${\rm Hg}(P)$ of our divisor $P$ in $e_\p\bigl({\widetilde J_m}(L_m)\otimes_\Z F_\p\bigr)$
is equal to the point $Q_{\chi,\Pi}$ considered in \cite[Theorem 4.6.2]{Ho3}, hence 
\[ {\rm Hg}(P)\not=0\;\Longleftrightarrow\;L'_K(f_\p,\chi_\p,1)\not=0, \]
by \cite[Theorem 4.6.2]{Ho3}. It remains to show that ${\rm Hg}(P)\neq0$ if and only if $\mathfrak Z_\p\neq0$, or equivalently, thanks to Corollary \ref{prop4.1}, if and only if $\mathfrak X_\p\neq0$. 

To prove this, we start by noticing that our hypotheses ensure that $f_\p$ is a newform of level $Np^m$, so that Eichler--Shimura theory and multiplicity one show that the summand $e_\p(\T_m\otimes_{\mathcal O_F}F_\p)=e_\p\bigl(\Ta_p({\widetilde J_m})\otimes_{\Z_p}F_\p\bigr)^\ord$ is sent isomorphically onto $V_\p$ under the map in \eqref{factor}. Using \cite[Corollary 2.2 and Proposition 2.3]{Ta}, it can be checked that the Kummer map 
\[ \widetilde J_m(L_m)\otimes_\Z F_\p\longrightarrow  H^1\bigl(L_m,\Ta_p(\widetilde J_m)\otimes_{\Z_p}F_\p\bigr) \] 
is injective, hence the above identifications yield an injection
\[ e_\p\bigl({\widetilde J_m}(L_m)\otimes_\Z F_\p\bigr)^{\!\ord}\;\longmono\;H^1(L_m,V_\p)\simeq H^1(L_m,V_\p^\dagger). \] 
By definition, the image of ${\rm Hg}(P)$ under this map is equal to the image of $\mathfrak X_\p$ under the restriction map $H^1(K,V_\p^\dagger)\rightarrow H^1(L_m,V_\p^\dagger)$, which is injective as its kernel is an $F_\p$-vector space that is killed by $[L_m:K]$. This concludes the proof. \end{proof}

\begin{proposition} \label{prop4.3}
Suppose $\p\in\Gen_2(\theta)$ satisfies (1) in Definition \ref{def-gen}. 
\begin{itemize}
\item[(a)] $L_K(f_\p,s)$ vanishes to odd order at $s=1$. 
\item[(b)] $L'_K(f_\p,1)\neq0\Longrightarrow\mathfrak Z_\p\neq0$.  
\end{itemize}
\end{proposition}

\begin{proof} In this case $m=1$ and there is a normalized newform $f_\p^\sharp$ of level $\Gamma_0(N)$, with Fourier coefficients $a_n^\sharp(\p)$, whose $p$-stabilization is $f_\p$. Moreover, $\chi_\p$ is the trivial character ${\bf 1}$. If $\epsilon_K$ denotes the character of $K$ then the $L$-functions of $f_\p^\sharp$ and $f_\p$ are related by the equality 
\begin{equation} \label{p-stab}
\bigg(1-\frac{p^{1-s}}{a_p(\p)}\bigg)\cdot\bigg(1-\epsilon_K(p)\frac{p^{1-s}}{a_p(\p)}\bigg)\cdot L_K(f_\p^\sharp,s)=L_K(f_\p,s). 
\end{equation}
The extra Euler factors on the left do not vanish at $s=1$ because $a_p(\p)\neq\pm1$ (since $a_p(\p)$ is a root of the Hecke polynomial $X^2-a_p^\sharp(\p)X+p=0$, its absolute value is $\sqrt p$), therefore the $L$-functions of $f_\p^\sharp$ and $f_\p$ have the same order of vanishing at $s=1$, which is odd in light of the functional equation recalled in \cite[p. 829]{Ho3} (here we apply again the results of \cite{Ho3} with $\mathfrak s=p$, $\mathfrak c=p$, $\mathfrak m=N$). This proves part (a). 

By Corollary \ref{prop4.1}, to prove part (b) one can equivalently show that if $\mathfrak X_\p=0$ then $L'_K(f_\p^\sharp,1)=0$. Let $X_0$ and $X_1$ be the (compact) Shimura curves over $\Q$ associated with the Eichler orders $R_0$ and $R_1$, respectively, so that
\[ X_i(\C)=\widehat R_i^\times\big\backslash\bigl(\widehat B^\times\times\PP\bigr)\big/ B^\times \] 
for $i=0,1$. There are degeneracy maps $\alpha, \beta:X_1(\C)\rightarrow X_0(\C)$, with $\alpha$ the canonical projection and $\beta$ corresponding to the action of an element of norm $p$ of $R_1$ normalising $R_1^\times$  (see, e.g., \cite[\S 3.1]{LRV}).
These maps induce homomorphisms by Picard functoriality 
\[ \alpha^*,\beta^*:J_0\longrightarrow J_1 \] 
where $J_0$ and $J_1$ are the Jacobian varieties of $X_0$ and $X_1$, respectively. The maps $\alpha^*$ and $\beta^*$ are compatible with the action of the Hecke operators $T_n$ for $(n,Np)=1$ and the diamond operators $\langle d\rangle $ for $(d,Np)=1$. Write $\mathfrak h$ for the Hecke algebra generated over $\Z$ by these operators and denote $e_\p$ the idempotent in $\mathfrak h\otimes_{\mathbb Z}F_\p$ associated with $\p$, as in the proof of Proposition \ref{prop4.2}. Define  
\[ e_\p\Ta_p(J_i):=e_\p\bigl(\Ta_p(J_i)\otimes_{\Z_p}F_\p\bigr) \]  
for $i=0,1$. Since the form $f_\p$ is old at $p$, the idempotent $e_\p$ kills the $p$-new quotient $J_1/(\alpha^*(J_0)\oplus\beta^*(J_0))$ of $J_1$, hence by \cite[\S 1.7]{BD-Heegner} there is a decomposition  
\begin{equation} \label{deco-Tate}
e_\p\Ta_p(J_1)\simeq\alpha^*e_\p\Ta_p(J_0)\oplus \beta^*e_\p\Ta_p(J_0).
\end{equation}
The Hecke operator $T_n$ acts on $e_\p\Ta_p(J_0)$ as multiplication by $a_n^\sharp(\p)=a_n(\p)$ and, by strong multiplicity one, $T_p$ acts on it as multiplication by $a^\sharp_p(\p)$. On the other hand, we have the relations
\[ U_p\circ\alpha^*=p\beta^*,\qquad U_p\circ\beta^*=\beta^*\circ T_p-\alpha^*. \] 
A proof of the analogous formulas for classical modular curves is given in \cite[Lemma 2.1]{BD-ajm}, and the arguments carry over \emph{mutatis mutandis} to the case of Shimura curves attached to division quaternion algebras. For a precise reference, see \cite[p. 93]{Helm}.

It follows that $U_p$ acts on $e_\p\Ta_p(J_1)$ via decomposition \eqref{deco-Tate} with characteristic polynomial 
\[ \bigl(X^2-a_p^\sharp(\p)X+p\bigr)^2=(X-a_p(\p))^2\cdot(X-p/a_p(\p))^2 \] 
and is diagonalizable. Furthermore, the projection to the $a_p(\p)$-eigenspace corresponds to the ordinary projection, because the other eigenvalue $p/a_p(\p)$ has positive $p$-adic valuation. By multiplicity one, the $F_\p$-vector spaces $e_\p(\Ta_p^\ord({J_1})\otimes_{\Z_p}F_\p)$ and $e_\p(\T_1\otimes_{\mathcal O_F}F_\p)$ can be identified. Thus, as in the proof of Proposition \ref{prop4.2}, it follows from the Eichler--Shimura relations that the map $\T_1\otimes_{\mathcal O_F}F_\p\rightarrow V_\p$ arising from \eqref{factor} takes the summand $e_\p(\T_1\otimes_{\mathcal O_F}F_\p)$ isomorphically onto $V_\p$. Since, as observed in the proof of Proposition \ref{prop4.2}, the Kummer map 
\[ \widetilde J_1(L_1)\otimes_\Z F_\p\longrightarrow  H^1\bigl(L_1,\Ta_p(\widetilde J_1)\otimes_{\Z_p}F_\p\bigr) \] 
is injective, using the above identifications we get an injection
\begin{equation} \label{eq*}
e_\p\bigl({\widetilde J_1}(L_1)\otimes_\Z F_\p\bigr)^{\!\ord}\;\longmono\;H^1(L_1,V_\p).
\end{equation}
If ${\rm Hg}$ is the Hodge embedding used in the proof of Proposition \ref{prop4.2}, the map in \eqref{eq*} sends ${\rm Hg}(\PP_\p)$ to $\mathfrak X_\p$. Taking into account the previous description of the ordinary projection and the injectivity of \eqref{eq*}, it follows that 
\begin{equation} \label{criterion}
\mathfrak X_\p\neq0\;\Longleftrightarrow\;\bigg(U_p-\frac{p}{a_p(\p)}\bigg)\bigl(e_\p \widetilde Q\bigr)\not=0,
\end{equation}
where 
\[ \widetilde Q:=\sum_{\sigma\in\Gal(L_1/K)}\widetilde P_1^\sigma\in\Div\bigl(\widetilde X_1^{(K)}\bigr) \] 
maps to $\PP_\p$ under the ordinary projection (recall that the character $\chi_{0,\p}$ is trivial in this case). To complete the proof we shall relate $e_\p \widetilde Q$, and thus $\mathfrak X_\p$, to classical Heegner points on $X_0$. 

Let $P_1$ denote the image of $\widetilde P_1$ under the canonical map $\widetilde X_1\rightarrow X_1$. This is a Heegner point of conductor $p$ on the indefinite Shimura curve $X_1$ considered in \cite[Section 2]{BD-Heegner}, and lives in $X_1(H_1)$. Define also $Q$ to be the image of $\widetilde Q$ in $X_1$, so that 
\[ Q:=\sum_{\sigma\in\Gal(L_1/K)}P_1^\sigma=[L_1:H_1]P\in \Div\bigl(X_1(H_{1})\bigr) \]
with 
\[ P:=\sum_{\sigma\in\Gal(H_1/K)}P_1^\sigma\in\Div\bigl(X_1(H_1)\bigr). \]
Since $f_\p^\sharp$ has trivial character, by Faltings's isogeny theorem (\cite[\S 5, Korollar 2]{Fa}) there is an isomorphism
\[ e_\p\Ta_p(\widetilde J_1)\simeq e_\p\Ta_p(J_1) \] 
such that the Kummer maps identify the images of $[H_1:H]\cdot{\rm Hg}(e_\p P)$ and $e_\p{\rm Hg}(Q)$. Using \eqref{criterion}, we thus conclude that  
\begin{equation} \label{criterion1}
\mathfrak X_\p\neq0\;\Longleftrightarrow\;\bigg(U_p-\frac{p}{a_p(\p)}\bigg)(e_\p P)\neq0.
\end{equation}
Let $\alpha_*,\beta_*:J_1\rightarrow J_0$ be the maps between Jacobians induced from $\alpha$ and $\beta$ by Albanese (i.e., covariant) functoriality. The relation $\beta_* U_p=p\alpha_*$ combined with \eqref{criterion1} then shows that 
\begin{equation} \label{criterion2}
\mathfrak X_\p=0\;\Longrightarrow\;\bigg(\alpha_*-\frac{\beta_*}{a_p(\p)}\bigg)(e_\p P)= 0.
\end{equation}
Now we relate the vanishing of $e_\p P$ with special values of $L$-series. If $p$ splits in $K$ then let $\sigma$ and $\sigma^*$ denote the Frobenius elements in $\Gal(H_0/K)$ corresponding to the two primes of $K$ above $p$. Write $u$ for half the order of $\mathcal O_K^\times$ and consider the points $\alpha(P_1), \beta(P_1)\in X_0(H_1)$. By \cite[Proposition 4.6]{LV}, the point $\alpha(P_1)$ is a Heegner point of conductor $p$. Moreover, as already pointed out, the map $\beta$ is given by the composition of $\alpha$ with the Atkin--Lehner involution at $p$ (see, e.g., \cite[\S 3.1]{LRV}). Comparing with the constructions of \cite[\S\S 4.1--4.2]{LV} (see, in particular, \cite[eq. (12)]{LV}), it follows that $\beta(P_1)$ is a Heegner point of conductor $1$ living in $X_0(H_0)$. The Euler system relations in \cite[\S 2.4]{BD-Heegner} can then be applied, and we get
\begin{equation} \label{BD-eq}
u{\rm {tr}}_{H_1/H}\alpha(P_1)=\begin{cases}T_p\bigl(\beta( P_1)\bigr)&\text{if $p$ is inert in $K$}\\[2mm](T_p-\sigma-\sigma^*)\bigl(\beta(P_1)\bigr)&\text{if $p$ is split in $K$}\end{cases}
\end{equation} 
as divisors on $X_0$. Thanks to \eqref{BD-eq} and the relation $u[H_1:H]=p-\epsilon_K(p)$, one has 
\begin{equation} \label{alpha-beta}
\alpha_*(P)=\frac{T_p-1-\epsilon_K(p)}{p-\epsilon_K(p)}\bigl(\beta_*(P)\bigr).
\end{equation}
Combining \eqref{criterion2} and \eqref{alpha-beta} we see that if $\mathfrak X_\p=0$ then 
\[ \bigg(\frac{1}{a_p(\p)}-\frac{a_p^\sharp(\p)-1-\epsilon_K(p)}{p-\epsilon_K(p)}\bigg)\cdot\beta_*(P)=0. \]
Since $a_p(\p)$ is a root of the polynomial $X^2-a_p^\sharp(\p)X+p$ and $a_p(\p)\neq\pm1$, the coefficient multiplying $\beta_*(P)$ is nonzero, hence we conclude that if $\mathfrak X_\p$ is zero then $\beta_*(P)$ is zero too. However, thanks to \cite[Theorem C]{Zh-1}, this possibility is ruled out by our assumption that $L'_K(f_\p,1)\not=0$. \end{proof}

If $\p$ is an arithmetic prime of $\mathcal R$ of weight $2$ then we say that the pair $(f_\p,\chi_\p)$ has \emph{analytic rank one} if $L_K(f_\p,\chi_\p,1)=0$ and $L'_K(f_\p,\chi_\p,1)\neq0$. 

\begin{corollary} \label{equivalence}
The following are equivalent: 
\begin{enumerate}
\item there is a $\p\in\Gen_2(\theta)$ such that $(f_\p,\chi_\p)$ has analytic rank one;
\item there is an arithmetic prime $\p$ such that $\mathfrak Z_\p\neq0$; 
\item $\mathfrak Z$ is not $\mathcal R$-torsion; 
\item $\mathfrak Z_\p\neq0$ for all but finitely many arithmetic primes $\p$; 
\item $(f_\p,\chi_\p)$ has analytic rank one for all but finitely many $\p\in\Gen_2(\theta)$. 
\end{enumerate}
\end{corollary}

\begin{proof} The implication $(1)\Rightarrow(2)$ is immediate from Propositions \ref{prop4.2} and \ref{prop4.3}. Now write $\T_\p^\dagger$ for the localization of $\T^\dagger$ at an arithmetic prime $\p$ of $\mathcal R$, so that the triangle
\begin{equation} \label{triangle-eq}
\xymatrix{H^1(K,\T^\dagger)\ar[r]\ar[d]&H^1(K,V_\p^\dagger)\\H^1(K,\T_\p^\dagger)\ar[ru]&}
\end{equation}
is commutative. Observe that, localization being an exact functor, there is also a canonical identification $H^1(K,\T_\p^\dagger)=H^1(K,\T^\dagger)_\p$. As in \cite[\S 10.1]{LV}, define $\mathfrak a_\mathcal R$ to be the annihilator in $\mathcal R$ of the finitely generated torsion $\mathcal R$-module $\prod_{\ell|N^-}H^1(K_\ell,\T^\dagger)_{\text{tors}}$ (cf. \cite[Proposition 4.2.3]{Ne}), then choose a nonzero $\lambda\in\mathfrak a_\mathcal R$. Let $\Sel_{\text{Gr}}(K,\T^\dagger)\subset H^1(K,\T^\dagger)$ denote the Greenberg Selmer group as defined, e.g., in \cite[\S 5.6]{LV}. By \cite[Proposition 10.1]{LV}, $\lambda\cdot\mathfrak Z\in\Sel_{\text{Gr}}(K,\T^\dagger)$. According to \cite[Proposition 12.7.13.4 (iii)]{Ne} plus \cite[eq. (21)]{Ho1}, the localization $\Sel_{\text{Gr}}(K,\T^\dagger)_\p$ at an arithmetic prime $\p$ is free of finite rank over $\mathcal R_\p$, so if (2) holds then diagram \eqref{triangle-eq} implies that $\lambda\cdot\mathfrak Z$ has nonzero, hence non-torsion, image in $\Sel_{\text{Gr}}(K,\T^\dagger)_\p$ for some arithmetic prime $\p$. It follows that $\lambda\cdot\mathfrak Z$ is not $\mathcal R$-torsion, hence $\mathfrak Z$ is not $\mathcal R$-torsion as well. This shows that $(2)\Rightarrow(3)$. To prove that $(3)\Rightarrow(4)$ one can proceed exactly as in the proof of the implication $\text{(c)}\Rightarrow\text{(d)}$ in \cite[Corollary 5]{Ho2}. Finally, the fact that $(4)\Rightarrow(5)$ is a straightforward consequence of Proposition \ref{prop4.2}, while $(5)\Rightarrow(1)$ is obvious. \end{proof}

Now let $\widetilde H^1_f(K,\T^\dagger)$ be Nekov\'a\v{r}'s extended Selmer group (see \cite[Ch. 6]{Ne}). 

\begin{theorem} \label{teo-indef}
Suppose there is a $\p\in\Gen_2(\theta)$ such that $(f_\p,\chi_\p)$ has analytic rank one. Then $\widetilde H^1_f(K,\T^\dagger)$ is an $\mathcal R$-module of rank one.
\end{theorem}

\begin{proof} As in the proof of Corollary \ref{equivalence}, choose a nonzero $\lambda\in\mathfrak a_\mathcal R$. In the course of proving Corollary \ref{equivalence} we showed that the existence of an arithmetic prime $\p\in\Gen_2(\theta)$ such that $(g_\p,\chi_\p)$ has analytic rank one implies that the class $\lambda\cdot\mathfrak Z$ is not $\mathcal R$-torsion. In other words, under this analytic condition \cite[Conjecture 10.3]{LV} is true, and then the $\mathcal R$-module $\widetilde H^1_f(K,\T^\dagger)$ has rank one by \cite[Theorem 10.6]{LV}. \end{proof}

\section{Vanishing of special values}\label{sec5}

Let the quaternion algebra $B$ be \emph{definite} and write $\p$ for a weight $2$ arithmetic prime of $\mathcal R$ of level $\Phi_{m_\p}$ and character $\psi_\p$. Once $\p$ has been fixed, set $m:=m_\p$ and $\psi:=\psi_\p$.  

\subsection{Picard groups}

For every integer $t\geq0$ define the $\mathfrak h_t^\ord$-module 
\[ \J_t:=\mathcal O_F\bigl[U_t\backslash\widehat B^\times/B^\times\bigr]^\ord\simeq\Big(\Pic(\widetilde X_t)\otimes_\Z\mathcal O_F\!\Big)^{\!\ord} \]
and form the inverse limit 
\[ \J_\infty:=\invlim_t\J_t. \] 
Consider the $\mathcal R$-component $\J:=\J_\infty\otimes_{\mathfrak h_\infty^\ord}\mathcal R$ of $\J_\infty$ and define $\J_\p:=\J\otimes_\mathcal R F_\p$. By \cite[Proposition 9.1]{LV}, $\J_\p$ is a one-dimensional $F_\p$-vector space. Note that the canonical projection $\J\rightarrow\J_\p$ can be factored as
\begin{equation} \label{factor.1}
\J\longrightarrow\J_m\otimes_{\mathcal O_F}F_\p\xrightarrow{\pi_{m,\p}}\J_\p.
\end{equation} 
This follows from Hida's control theorem \cite[Theorem 12.1]{H} (or, rather, from its extension \cite[Proposition 3.6]{LV-preprint} to our present context). Using the natural maps \[\eta_t:\D_t\longrightarrow\J_t\] arising from Picard (i.e., contravariant) functoriality, we get a map 
\[ \eta_K:H^0(\mathcal G_K,\D^\dagger)\longrightarrow\mathbb J. \]
Recall the point $\mathcal P$ introduced in \eqref{def-P} and define 
\[ \mathfrak Z:=\eta_K(\mathcal P)\in\mathbb J. \] 
Let $\mathfrak Z_\p$ denote the image of $\mathfrak Z$ in $\J_\p$. The map $\eta_K$ induces a map
\[ \eta_\p:H^0(\mathcal G_K,\D^\dagger_\p)\longrightarrow\mathbb J_\p, \]
and the square
\begin{equation} \label{square2-eq}
\xymatrix{H^0(\mathcal G_K,\D^\dagger)\ar[r]^-{\eta_K}\ar[d]&\mathbb J\ar[d]\\H^0(\mathcal G_K,\D^\dagger_\p)\ar[r]^-{\eta_\p}&\mathbb J_\p} 
\end{equation}
is commutative, so that $\mathfrak Z_\p$ is also equal to $\eta_\p(\mathcal P_\p)$ (in the notations of \S \ref{sec3.4}). We may also consider the class $\eta_m(\PP_\m^{\chi_\p})\in H^0(K,\J_m\otimes_{\mathcal O_F}F_\p)$, where $\PP_m^{\chi_\p}\in H^0(K,\D^\dagger_m\otimes_{\mathcal O_F}F_\p)$ is the element introduced in \eqref{pp-eq} and $\eta_m$ is extended by $F_p$-linearity, and define $\mathfrak X_\p$ to be the image of $\eta_m(\PP_m^{\chi_\p})$ in $H^1(K,\J_\p)$ via the map $\pi_{m,\p}$ appearing in \eqref{factor.1}. 

The following result is the analogue in the definite case of Proposition \ref{lemma1}. 

\begin{proposition} \label{lemma2}
The formula $\mathfrak X_\p=a_p(\p)^m[L_m:H_{m}]\mathfrak Z_\p$ holds in $H^0(\mathcal G_K,\J_\p)$.
\end{proposition}

\begin{proof} We argue as in the proof of Proposition \ref{lemma1}. Thanks to factorization \eqref{factor.1},  the class $\eta_\p(\mathcal P_\p)=\mathfrak Z_\p$ is equal to the  corestriction from $H$ to $K$ of the image of $U_p^{-m}\eta_m(\mathcal P_m)\in H^0(\mathcal G_H,\J_m)$ in $H^0(\mathcal G_K,\J_\p)$ under the map induced by the map labeled $\pi_{m,\p}$ in \eqref{factor.1}. Since the action of $\mathfrak h_\infty^\ord$ on $\J_\p$ is via the character $\phi_\p$ defined in \eqref{char-p}, one has   
\[ \eta_\p(\mathcal P_\p)=\pi_{m,\p}\Big({\rm cor}_{H/K}\bigl(a_p(\p)^{-m}\eta_m(\mathcal P_m)\bigr)\!\Big). \]
Finally, using \eqref{2} we find that 
\[ \eta_\p(\mathcal P_\p)=a_p(\p)^{-m}[L_m:H_m]^{-1}\pi_{m,\p}\bigl(\eta_m(\mathbb P_m^{\chi_\p})\bigr)=a_p(\p)^{-m}[L_m:H_m]^{-1}\mathfrak X_\p, \]
and the result follows. \end{proof} 

We conclude this subsection with the counterpart of Corollary \ref{prop4.1}.

\begin{corollary} \label{prop5.1}
$\mathfrak Z_\p\neq0\Longleftrightarrow\mathfrak X_\p\neq0$.
\end{corollary}

\begin{proof} Immediate from Proposition \ref{lemma2}. \end{proof}

\subsection{Non-vanishing results} 

Keep the notation of \S \ref{sec4.2} for $L_K(f_\p,s)$ and $L_K(f_\p,\chi_\p,s)$. 

\begin{proposition} \label{prop5.2}
Suppose $\p\in\Gen_2(\theta)$ satisfies either (2) or (3) in Definition \ref{def-gen}. 
\begin{itemize}
\item[(a)] $L_K(f_\p,\chi_\p,s)$ vanishes to even order at $s=1$.
\item[(b)] $L_K(f_\p,\chi_\p,1)\neq0\Longleftrightarrow\mathfrak Z_\p\neq0$. 
\end{itemize} 
\end{proposition}

\begin{proof} As in the proof of Proposition \ref{prop4.2}, let $\mathfrak h$ denote the algebra generated over $\Z$ by the Hecke operators $T_n$ with $(n,Np)=1$ and the diamond operators $\langle d\rangle$ with $(d,Np)=1$ acting on the $\C$-vector space of cusp forms in $S_2(\Phi_m,\C)$. Moreover, write $e_\p$ for the idempotent in $\mathfrak h\otimes_\Z F_\p$ which projects onto the maximal summand on which every $T_n$ acts as $a_\p(n)$ and $\langle\cdot\rangle$ acts as $\chi_{0,\p}^{-1}$. 

We apply the results in \cite{Ho3} with $\mathfrak s=p^m$, $\mathfrak c=p^m$ and $\mathfrak m=N$. As pointed out in \cite[p. 829]{Ho3}, the sign of the functional equation for $L_K(f_\p,\chi_\p,s)$ is $+1$, and this implies part (a). 
 
Let us prove (b). Define 
\[ \widetilde J_m:=\mathcal O_F\bigl[U_m\backslash\widehat B^\times/B^\times\bigr]\simeq\Pic(\widetilde X_m)\otimes_\Z\mathcal O_F\] and consider the divisor 
\[ P:=\sum_{\sigma\in\Gal(L_m/K)}\PP_m^\sigma\otimes\chi_\p^{-1}(\sigma)\in\Div\bigl(\widetilde X_m^{(K)}\bigr)
\otimes_\Z F_\p.\] 
Write $Q$ for the image of $P$ in $\widetilde J_m$. After fixing an appropriate embedding $K\hookrightarrow B$, the divisor $Q$ is equal to the divisor $Q_\chi$ considered in \cite[Theorem 3.3.3]{Ho3}. Let $\phi_\p$ be the (unique up to nonzero multiplicative factors) modular form on $B$ associated with $f_\p$ by the Jacquet--Langlands correspondence. Up to rescaling, one may view 
$\phi_\p$ as an $F_\p$-valued function on $U_m\backslash\widehat B^\times/B^\times$. By a slight abuse of notation, we adopt the same symbol for the $F_\p$-linear extension of $\phi_\p$ to $\widetilde J_m\otimes_{\mathcal O_F}F_\p$. By \cite[Theorem 3.3.3]{Ho3}, we know that 
\[ \phi_\p(Q)\not=0\;\Longleftrightarrow\;L_K(f_\p,\chi_\p,1)\not=0, \]
so it remains to show that $\phi_\p(Q)\not=0$ if and only if $\mathfrak Z_\p\not=0$, or equivalently, thanks to Corollary \ref{prop5.1}, if and only if $\mathfrak X_\p\neq0$. 

To prove this, we compute $\phi_\p(Q)$ as follows. Let $e_\p Q$ be the image of $Q$ in $e_\p(\widetilde J_m\otimes_{\mathcal O_F}F_\p)$. Since $f_\p$ is a newform of level $Np^m$, the same is true of $\phi_\p$ and it follows, by multiplicity one, that 
$e_\p(\widetilde J_m\otimes_{\mathcal O_F}F_\p)$ is one-dimensional. Therefore $\phi_\p(Q)\not=0$ if and only if $e_\p Q\not=0$. 
Now the image of $P$ in $\D_\p$ is $\PP_\p$, hence the image of $Q$ in $\J_\p$ is $\mathfrak X_\p$, by the commutativity of \eqref{square2-eq}. Finally, notice that, by multiplicity one (cf. \cite[\S 1.9]{BD-Heegner}), there is an isomorphism $e_\p(\widetilde J_m\otimes_{\mathcal O_F}F_\p)\simeq\J_\p$ sending $e_\p Q$ to $\mathfrak X_\p$, which completes the proof. \end{proof}

\begin{proposition} \label{prop5.3}
Suppose $\p\in\Gen_2(\theta)$ satisfies (1) in Definition \ref{def-gen}. 
\begin{itemize}
\item[(a)] $L_K(f_\p,s)$ vanishes to even order at $s=1$. 
\item[(b)] $L_K(f_\p,1)\neq0\Longrightarrow\mathfrak Z_\p\neq0$.  
\end{itemize}\end{proposition}

\begin{proof} Part (a) follows as in the proof of Proposition \ref{prop4.3}, this time observing that, as explained in \cite[p. 829]{Ho3}, the order of vanishing of $L_K(f_\p^\sharp,s)$ at $s=1$ is even (here we apply again the results of \cite{Ho3} with $\mathfrak s=p$, $\mathfrak c=p$, $\mathfrak m=N$), equal to that of $L_K(f_\p,s)$ by \eqref{p-stab}. 

By Corollary \ref{prop5.1}, proving (b) is equivalent to proving that if $\mathfrak X_\p=0$ then $L_K(f_\p^\sharp,1)=0$. Let $X_0$ and $X_1$ be the Shimura curves over $\Q$ associated with the Eichler orders $R_0$ and $R_1$, respectively, so that (with notation as in \S \ref{definite-shimura-subsec}) one has
\[ X_i=\widehat R_i^\times\backslash(\widehat B^\times\times\PP)/ B^\times \] 
for $i=0,1$. There are degeneracy maps $\alpha,\beta:X_1\rightarrow X_0$ where $\alpha$ is the canonical projection and $\beta$ is the composition of $\alpha$ with an Atkin--Lehner involution at $p$ (see, e.g., \cite[\S 1.7]{BD-Heegner}). The maps $\alpha,\beta$ induce by Picard functoriality homomorphisms 
\[ \alpha^*,\beta^*:J_0\longrightarrow J_1 \] 
where $J_i$ is the Picard group of $X_i$ for $i=0,1$, so that $J_i=\Z[\widehat R_i^\times\backslash\widehat B^\times/B^\times]$. 
The maps $\alpha^*$ and $\beta^*$ are compatible with the action of the Hecke operators $T_n$ for $(n,Np)=1$ and the diamond operators $\langle d\rangle $ for $(d,Np)=1$. Write $\mathfrak h$ for the Hecke algebra generated over $\Z$ by these operators and denote $e_\p$ the idempotent in $\mathfrak h\otimes_{\mathbb Z}F_\p$ associated with $\p$, as in the proof of Proposition \ref{prop5.2}. Define  
\[ e_\p J_i:=e_\p(J_i\otimes_\Z F_\p) \]  
for $i=0,1$. Since the form $f_\p$ is old at $p$, the idempotent $e_\p$ annihilates the $p$-new quotient $J_1/(\alpha^*(J_0)\oplus\beta^*(J_0))$ of $J_1$, hence by \cite[\S 1.7]{BD-Heegner} there is a decomposition 
\begin{equation} \label{deco-def}
e_\p J_1\simeq\alpha^*e_\p J_0\oplus \beta^*e_\p J_0.
\end{equation}
The operator $T_n$ acts on $e_\p J_0$ as multiplication by $a_n^\sharp(\p)=a_n(\p)$ and, by strong multiplicity one, $T_p$ acts on it as multiplication by $a^\sharp_p(\p)$. On the other hand, we have the relations 
\[ U_p\circ\alpha^*=p\beta^*,\qquad U_p\circ\beta^*=\beta^*\circ T_p-\alpha^*. \] 
For a proof, see \cite[Theorem 3.16]{JL}. It follows that $U_p$ acts on $e_\p J_1$ via decomposition \eqref{deco-def} with characteristic polynomial 
\[ \bigl(X^2-a_p^\sharp(\p)X+p\bigr)^2=(X-a_p(\p))^2\cdot(X-p/a_p(\p))^2 \] 
and is diagonalizable. Furthermore, the projection to the $a_p(\p)$-eigenspace corresponds to the ordinary projection, because the other eigenvalue $p/a_p(\p)$ has positive $p$-adic valuation. By multiplicity one, the $F_\p$-vector spaces $e_\p J_1$ and $e_\p\widetilde J_1$ can be identified. Thus the map $\J_1\otimes_{\mathcal O_F}F_\p\rightarrow \J_\p$ arising from \eqref{factor.1} takes the summand $e_\p(\J_1\otimes_{\mathcal O_F}F_\p)$ isomorphically onto $\J_\p$. Taking the previous description of the ordinary projection into account, it follows that 
\begin{equation} \label{criterion-def}
\mathfrak X_\p\neq0\;\Longleftrightarrow\;\bigg(U_p-\frac{p}{a_p(\p)}\bigg)\bigl(e_\p\widetilde Q\bigr)\not=0,
\end{equation}
where 
\[ \widetilde Q:=\sum_{\sigma\in\Gal(L_1/K)}\widetilde P_1^\sigma\in\Div\bigl(\widetilde X_1^{(K)}\bigr) \] 
maps to $\PP_\p$ under the ordinary projection (recall that the character $\chi_{0,\p}$ is trivial in this case). To complete the proof we shall relate $e_\p \widetilde Q$, and thus $\mathfrak X_\p$, to classical Heegner (or, rather, Gross--Heegner) points on $X_0$. 

Let $P_1$ denote the image of $\widetilde P_1$ under the canonical map $\widetilde X_1\rightarrow X_1$. This is a Heegner point of conductor $p$ on the definite Shimura curve $X_1$ considered in \cite[Section 2]{BD-Heegner}. Define also $Q$ to be the image of $\widetilde Q$ in $X_1$, so that 
\[ Q:=\sum_{\sigma\in\Gal(L_1/K)}P_1^\sigma=[L_1:H_1]P\in \Div\bigl(X_1^{(K)}\bigr) \]
with 
\[ P:=\sum_{\sigma\in\Gal(H_1/K)}P_1^\sigma\in\Div\bigl(X_1^{(K)}\bigr). \]
Thus, using \eqref{criterion-def}, we conclude that  
\begin{equation} \label{criterion1-def}
\mathfrak X_\p\neq0\;\Longleftrightarrow\;\bigg(U_p-\frac{p}{a_p(\p)}\bigg)(e_\p P)\neq0.
\end{equation}
Let $\alpha_*,\beta_*:J_1\rightarrow J_0$ be the maps between Jacobians induced from $\alpha$ and $\beta$ by Albanese functoriality. The relation $\beta_* U_p=p\alpha_*$ combined with \eqref{criterion1-def} then shows that 
\begin{equation} \label{criterion2-def}
\mathfrak X_\p=0\;\Longrightarrow\;\bigg(\alpha_*-\frac{\beta_*}{a_p(\p)}\bigg)(e_\p P)= 0.
\end{equation}
Now we relate $e_\p P$ with special values of $L$-series. If $p$ splits in $K$ then let $\sigma$ and $\sigma^*$ denote the Frobenius elements in $\Gal(H_0/K)$ corresponding to the two primes of $K$ above $p$. Write $u$ for half the order of $\mathcal O_K^\times$ and consider the points $\alpha(P_1),\beta(P_1)\in X_0(K)$. By \cite[Proposition 4.6]{LV}, $\alpha(P_1)$ is a Heegner point of conductor $p$. Moreover, as already pointed out, the map $\beta$ is given by the composition of $\alpha$ with the Atkin--Lehner involution at $p$. Comparing with the constructions of \cite[\S\S 4.1--4.2]{LV} (see, in particular, \cite[eq. (12)]{LRV}), it follows that $\beta(P_1)$ is a Heegner point of conductor $1$. The Euler system relations in \cite[\S 2.4]{BD-Heegner} can then be applied, and we get that 
\begin{equation} \label{BD-eq-def}
u{\rm {tr}}_{H_1/H}\alpha(P_1)=\begin{cases}T_p\bigl(\beta(P_1)\bigr)&\text{if $p$ is inert in $K$}\\[2mm](T_p-\sigma-\sigma^*)\bigl(\beta(P_1)\bigr)&\text{if $p$ is split in $K$}\end{cases}
\end{equation} 
as divisors on $X_0$. Thanks to \eqref{BD-eq-def} and the relation $u[H_1:H]=p-\epsilon_K(p)$, one has 
\begin{equation} \label{alpha-beta-def}
\alpha_*(P)=\frac{T_p-1-\epsilon_K(p)}{p-\epsilon_K(p)}\bigl(\beta_*(P)\bigr).
\end{equation}
Combining \eqref{criterion2-def} and \eqref{alpha-beta-def} we see that if $\mathfrak X_\p=0$ then 
\[ \bigg(\frac{1}{a_p(\p)}-\frac{a_p^\sharp(\p)-1-\epsilon_K(p)}{p-\epsilon_K(p)}\bigg)\cdot\beta_*(e_\p P)=0. \]
Since $a_p(\p)$ is a root of the polynomial $X^2-a_p^\sharp(\p)X+p$ and $a_p(\p)\neq\pm1$, the coefficient multiplying $\beta_*(e_\p P)$ is nonzero, hence we conclude that if $\mathfrak X_\p$ is zero then $\beta_*(e_\p P)=e_\p\beta_*(P)$ is zero too. However, thanks to \cite[Theorem 1.3.2]{Zh}, this possibility is ruled out by the non-vanishing condition $L_K(f_\p^\sharp,1)\not=0$, 
which is equivalent to our assumption $L_K(f_\p,1)\not=0$ (cf. \eqref{p-stab}). Indeed, write $\phi_\p^\sharp$ for the (unique up to nonzero multiplicative factors) modular form on $B$ attached to $f_\p^\sharp$ by the Jacquet--Langlands correspondence. At the cost of rescaling, we may view $\phi_\p^\sharp$ as an $F_\p$-valued function on $\widehat R^\times\backslash\widehat B^\times/B^\times$, and we use the symbol $\phi_\p$ also for its $F_\p$-linear extension to $J_0\otimes_\Z F_\p$. Then \cite[Theorem 1.3.2]{Zh} shows that $\phi_\p^\sharp(\beta_*(P))\not=0$ if and only if $L_K(f_\p^\sharp, 1)\not=0$. But $\phi_\p^\sharp(\beta_*(P))=e_\p\beta_*(P)$, and the result follows. \end{proof}

If $\p$ is an arithmetic prime of $\mathcal R$ of weight $2$ then we say that the pair $(f_\p,\chi_\p)$ has \emph{analytic rank zero} if $L_K(f_\p,\chi_\p,1)\neq0$. 

\begin{corollary} \label{coro-def}
The following are equivalent: 
\begin{enumerate}
\item there is a $\p\in\Gen_2(\theta)$ such that $(f_\p,\chi_\p)$ has analytic rank zero; 
\item there is an arithmetic prime $\p$ such that $\mathfrak Z_\p\neq0$; 
\item $\mathfrak Z$ is not $\mathcal R$-torsion; 
\item $\mathfrak Z_\p\neq 0$ for all but finitely many arithmetic primes $\p$; 
\item $(f_\p,\chi_\p)$ has analytic rank zero for all but finitely many primes $\p\in\Gen_2(\theta)$. 
\end{enumerate}
\end{corollary}

\begin{proof} The implication $(1)\Rightarrow(2)$ is immediate from Propositions \ref{prop5.2} and \ref{prop5.3}. On the other hand, the localization of $\J$ at $\p$ is free of rank one over $\mathcal R_\p$ by \cite[Proposition 9.1]{LV}, and the implication $(2)\Rightarrow(3)$ follows from this. If (3) holds then, since $\J$ is finitely generated over $\mathcal R$, \cite[Lemma 2.1.7]{Ho1} shows that $\mathfrak Z\not\in\p(\J\otimes_\mathcal R\mathcal R_\p)$ for all but finitely many $\p$, which proves $(4)$. Finally, $(4)\Rightarrow(5)$ is a direct consequence of Proposition \ref{prop5.2}, while $(5)\Rightarrow(1)$ is obvious. \end{proof}

 

To conclude this subsection, we prove 

\begin{theorem} \label{teo-def} 
Suppose there is a $\p\in\Gen_2(\theta)$ such that $(f_\p,\chi_\p)$ has analytic rank zero. Then $\widetilde H^1_f(K,\T^\dagger)$ is $\mathcal R$-torsion.
\end{theorem}

\begin{proof} Combining our assumption and Corollary \ref{coro-def}, we see that $(f_\p,\chi_\p)$ has analytic rank zero for all but finitely many primes $\p\in\Gen_2(\theta)$. For every arithmetic prime $\p$ of $\mathcal R$ let $H_f^1(K,V_\p^\dagger)$ be the Bloch--Kato Selmer group, whose definition (in terms of Fontaine's ring $B_{\text{cris}}$) can be found in \cite[Sections 3 and 5]{BK}. A result of Nekov\'a\v{r} (\cite[Theorem A]{Nek-Level}) then shows that $H^1_f(K,V_\p^\dagger)=0$ for infinitely many $\p\in\Gen_2(\theta)$. (Observe that, in the notation of \cite{Nek-Level}, $\chi=\,^c\!\chi$ is our character $\chi_\p$, $\omega$ is our character $\chi_{0,\p}^{-1}$ and $V\otimes\chi$ is our Galois representation $V_\p^\dagger$.) Now one argues as in the proof of \cite[Theorem 9.8]{LV}. Briefly, as explained in the proof of \cite[Corollary 3.4.3]{Ho1}, for every arithmetic prime $\p$ of $\mathcal R$ there is an injection 
\begin{equation} \label{injection-nekovar-eq}
\widetilde H^1_f(K,\T^\dagger)_\p\big/\p\widetilde H^1_f(K,\T^\dagger)_\p\;\longmono\;\widetilde H^1_f(K,V^\dagger_\p). 
\end{equation}
By \cite[Proposition 5.5]{LV}, for every $\p\in\Gen_2(\theta)$ there is an isomorphism 
\[ \widetilde H^1_f(K,V^\dagger_\p)\simeq H^1_f(K,V^\dagger_\p), \]
hence the left hand side in \eqref{injection-nekovar-eq} is trivial for infinitely many $\p$. By \cite[Lemma 2.1.7]{Ho1}, this shows that $\widetilde H^1_f(K,\T^\dagger)$ is $\mathcal R$-torsion. \end{proof}

\section{Low rank results over arbitrary quadratic fields} \label{complementary-subsec}

In this final section we give some complementary results in the spirit of Theorems \ref{teo-indef} and \ref{teo-def} but whose proofs do not use the elements $\mathfrak Z$ introduced before and work for any quadratic number field $K$ (not necessarily imaginary). In fact, these results are essentially corollaries of those obtained by Howard in \cite[Section 4]{Ho2}, and we add them here for completeness. 

\subsection{Splitting the Selmer groups}

Fix a quadratic field $K$ of discriminant $D$ prime to $Np$ and let $\epsilon_K$ denote the character of $K/\Q$. For each arithmetic prime $\p$ of $\mathcal R$ consider the twisted Galois representation $V_\p\otimes\epsilon_K$ and the twisted modular form
\[ f_\p\otimes\epsilon_K=\sum_{n\geq 1}\epsilon_K(n)a_n(\p)q^n. \] 
Since $D$ is prime to $Np$, the form $f_\p\otimes\epsilon_K$ is again a $p$-stabilized newform, of weight $k_\p$, level $ND^2p$ and character $\vartheta_\p\epsilon_K^2:=\psi_\p\omega^{k+j-k_\p}\epsilon_K^2$, where $k_\p$ is the weight of $f_\p$ and we denote $\vartheta_\p:=\psi_\p\omega^{k+j-k_\p}$ the character of $f_\p$ (so $\vartheta_\p=\theta_\p^2=\chi_{0,\p}^{-1}$ if $\p$ has weight $2$). See \cite[Theorem 9]{Li} for a proof of this fact (see also the discussion in the introduction to \cite{AL}). Since $\epsilon_K$ is quadratic, we shall simply write $\vartheta_\p$ for $\vartheta_\p\epsilon_K^2$, with the convention that $\vartheta_\p$ is viewed as a character of conductor $Dp^{m_\p}$ via the canonical projection $(\Z/Dp^{m_\p}\Z)^\times\rightarrow (\Z/p^{m_\p}\Z)^\times$. It follows that $V_\p\otimes\epsilon_K$ is the $p$-adic Galois representation attached to $f_\p\otimes\epsilon_K$ and $V_\p^\dagger\otimes\epsilon_K$ is the representation attached to $f_\p^\dagger\otimes\epsilon_K$.  

Now let $g$ be a weight $2$, $p$-ordinary normalized eigenform on $\Gamma_0(Mp)$ for some integer $M\geq1$ prime to $p$ with (non-necessarily primitive) Dirichlet character $\psi$, and let $\tilde\psi$ denote the primitive character associated with $\psi$. We say that $g$ has an \emph{exceptional zero} if $\tilde\psi(p)\neq 0$ and the $p$-th Fourier coefficient of $g$ (which is equal to the eigenvalue of $U_p$ acting on $g$) is $\pm1$. See \cite[Ch. I, \S 15]{MTT} for more general definitions and details. 

Since we will need it in the following, we recall the definition of the (strict) Greenberg Selmer group. Suppose that $L$ is a number field, $v$ is a finite place of $L$ and $D_v\subset\Gal(\bar\Q/\Q)$ is a decomposition group at $v$. Let $M$ be a $\Gal(\bar\Q/\Q)$-module, which we require, for simplicity, to be a two-dimensional vector space over a finite field extension $E$ of $\Q_p$. We also assume that $M$ is ordinary at all places $v|p$, which means that for all $v|p$ there is a short exact sequence
\begin{equation} \label{ordinary}
0\longrightarrow M^+_v\longrightarrow M\overset\pi\longrightarrow M^-_v\longrightarrow0
\end{equation} 
of $D_v$-modules such that $M^+_v$ and $M^-_v$ are one-dimensional $E$-vector spaces and the inertia subgroup $I_v\subset D_v$ acts trivially on $M^-_v$. For $v\nmid p$ the Greenberg local conditions at $v$ are given by the group 
\[ H^1_{\rm Gr}(L_v,M):=\ker\Big(H^1(L_v,M)\longrightarrow H^1(L_v^{\rm unr},M)\Big) \] 
where $L_v^{\rm unr}$ is the maximal unramified extension of the completion $L_v$ of $L$ at $v$, while for $v\mid p$ the Greenberg local conditions at $v$ are defined as 
\[ H^1_{\rm Gr}(L_v,M):=\ker\Big(H^1(L_v,M)\overset\pi\longrightarrow H^1(L_v,M^-_v)\Big); \] 
here, by a slight abuse of notation, we write $\pi$ for the map in cohomology induced by the corresponding map in \eqref{ordinary}.

The (strict) Greenberg Selmer group $\Sel_{\rm Gr}(L,M)$ is defined as 
\[ \Sel_{\rm Gr}(L,M):=\ker\Big(H^1(L,M)\xrightarrow{\prod_v\res_v}\prod_vH^1(L_v,M)\big/H^1_{\rm Gr}(L_v,M)\Big), \] 
where the product is over all finite places of $L$ and the maps $\res_v$ are the usual restriction maps. The groups $\Sel_{\rm Gr}(L,M)$ and $\widetilde H^1_f(L,M)$ sit in an exact sequence  
\[ 0\rightarrow\widetilde H^0_f(L,M)\longrightarrow H^0(L,M)\longrightarrow\bigoplus_{v|p}H^0(L_v,M^-_v)\longrightarrow\widetilde H^1_f(L,M)\longrightarrow\Sel_{\rm Gr}(L,M)\rightarrow0 \]
(see \cite[Lemma 9.6.3]{Ne}).

For $L=\Q$ or $L=K$ and $v|p$, we know that there is an exact sequence  
\begin{equation} \label{ordinary1}
0\longrightarrow(V_\p^\dagger)_v^+\longrightarrow V_\p^\dagger\longrightarrow(V_\p^\dagger)_v^-\longrightarrow0
\end{equation} 
as \eqref{ordinary} above for $M=V_\p^\dagger$. When $L=\Q$ we can twist \eqref{ordinary1} by $\epsilon_K$ to get a $D_p$-equivariant short exact sequence 
\begin{equation} \label{ordinary2}
0\longrightarrow(V_\p^\dagger\otimes\epsilon_K)_p^+\longrightarrow V_\p^\dagger\otimes\epsilon_K\longrightarrow(V_\p^\dagger\otimes\epsilon_K)_p^-\longrightarrow0
\end{equation} 
where $(V_\p^\dagger\otimes\epsilon_K)_p^\pm:=(V_\p^\dagger)_p^\pm\otimes\epsilon_K$. Furthermore, since $p$ is unramified in $K$, we have $\epsilon_K(\sigma)=1$ for all $\sigma\in I_v$. Therefore \eqref{ordinary2} is an exact sequece of type \eqref{ordinary} for $V_\p^\dagger\otimes\epsilon_K$, hence we may define $\Sel_{\rm Gr}(\Q,V_\p^\dagger\otimes\epsilon_K)$ by the previous recipe. 
 
If $f_\p$ does not have an exceptional zero then \cite[eq. (21)]{Ho1} ensures that  
\[ \widetilde H^1_f(L,V_\p^\dagger)\simeq\Sel_{\rm Gr}(L,V_\p^\dagger) \]
for $L=\Q$ or $L=K$. Following \cite[Lemma 2.4.4]{Ho1}, we extend this isomorphism to the case of the twisted representation $V_\p^\dagger\otimes\epsilon_K$. 

\begin{lemma} \label{lemma5.7}
If $f_\p\otimes\epsilon_K$ does not have an exceptional zero then 
\[ \widetilde H^1_f(\Q,V_\p^\dagger\otimes\epsilon_K)\simeq\Sel_{\rm Gr}(\Q,V_\p^\dagger\otimes\epsilon_K). \] 
\end{lemma}

\begin{proof} For simplicity, put $N:=(V_\p^\dagger\otimes\epsilon_K)_p^-$. We only need to show that $H^0(\Q_p,N)=0$. Recall that $D_v$ acts on $V_\p^\dagger$ by 
$\eta\Theta_\p^{-1}$ where $\eta:D_p/I_p\rightarrow F_\p$ sends the arithmetic Frobenius to the eigenvalue $\alpha_\p$ of $U_p$ acting on $f_\p$. Hence $D_p$ acts on the twist $N$ by $\eta\epsilon_K\Theta_\p^{-1}$. If $H^0(\Q_p,N)\neq0$ then $\eta\epsilon_K\Theta_\p^{-1}$ is trivial on $D_p$.  Since $\Theta_\p$ factors through $\Gal(\Q_p(\boldsymbol\mu_{p^\infty})/\Q_p)$, it follows that $\eta\epsilon_K$ is trivial on $\Gal(\bar\Q_p/\Q_p(\boldsymbol\mu_{p^\infty}))$. On the other hand, $\eta\epsilon_K$ is unramified, so, since it factors through the maximal abelian extension of $\Q_p$, it must be trivial. Thus $\alpha_\p=\epsilon_K(p)$ and $\Theta_\p$ is trivial too. Now we have 
\[ \Theta_\p=\epsilon_{\rm tame}^{(k+j-2)/2}\cdot\epsilon_{\rm wild}^{(k_\p-2)/2}\cdot\bigl(\psi_\p\circ\epsilon_{\rm wild}^{1/2}\bigr), \]
so if $\Theta_\p$ is trivial then $k_\p=2$ necessarily. Finally, the fact that $\Theta_\p$ is trivial implies that $\theta_\p$ is trivial too, hence $\vartheta_\p$  is trivial and the character of $f_\p\otimes\epsilon_K$ is $\epsilon_K$. Since $p$ is unramified in $K$, $\epsilon_K(p)\neq0$. Therefore, by definition, $f_\p\otimes\epsilon_K$ has an exceptional zero. \end{proof}

For lack of a convenient reference, we also add a proof of the following result, which is certainly known. 

\begin{proposition} \label{prop5.8}
If $\p$ is an arithmetic prime of $\mathcal R$ such that neither $f_\p$ nor $f_\p\otimes\epsilon_K$ has an exceptional zero then there is a decomposition 
\[ \widetilde H^1_f(K,V^\dagger_\p)=\widetilde H^1_f(\Q,V^\dagger_\p)\oplus\widetilde H^1_f(\Q,V^\dagger_\p\otimes\epsilon_K). \] 
\end{proposition}

\begin{proof} Our fixed embeddings $\bar\Q\hookrightarrow\bar\Q_\ell$ allow us to define for every finite place $v$ of $L=\Q$ or $L=K$ the decomposition group $D_v\subset\Gal(\bar \Q/L)$ and its inertia subgroup $I_v\subset D_v$. Now let $v$ be a finite place of $L$ and let $\ell$ denote the rational prime over which $v$ lies. Thus $D_v\subset D_\ell$ with index $2$ if $\ell$ does not split in $K$ and $D_v= D_\ell$ otherwise. If $\ell$ is split in $K$ and $\ell=v\bar v$ is a factorization of $\ell$ into prime ideals $v$ and $\bar v$ of the ring of integers of $K$ then $D_v= \tau^{-1}D_{\bar v}\tau$ where $\tau$ is the generator of $\Gal(K/\Q)$. Moreover, $I_v\subset I_\ell$ with index $2$ if $\ell$ ramifies $K$ and $I_v=I_\ell$ otherwise. Since, by assumption, neither $f_\p$ nor $f_\p\otimes\epsilon_K$ has an exceptional zero, \cite[eq. (21)]{Ho1} and Lemma \ref{lemma5.7} ensure that  
\[ \widetilde H^1_f(L,V_\p^\dagger)\simeq\Sel_{\rm Gr}(L,V_\p^\dagger),\quad\widetilde H^1_f(\Q,V_\p^\dagger\otimes\epsilon_K)\simeq\Sel_{\rm Gr}(\Q,V_\p^\dagger\otimes\epsilon_K) \] 
for $L=\Q$ or $L=K$. A direct computation using the explicit description given above shows that $\Sel_{\rm Gr}(K,V_\p^\dagger)$ is stable under the action of $\Gal(K/\Q)$. Since $F_\p$ is a field of characteristic zero and $\tau$ is an involution on $\Sel_{\rm Gr}(K,V^\dagger_\p)$, there is a decomposition 
\[ \Sel_{\rm Gr}(K,V^\dagger_\p)=\Sel_{\rm Gr}(K,V^\dagger_\p)^+\oplus\Sel_{\rm Gr}(K,V^\dagger_\p)^- \] 
where $\Sel_{\rm Gr}(K,V^\dagger_\p)^\pm$ is the $\pm$-eigenspaces for $\tau$. One has 
\[ \Sel_{\rm Gr}(K,V^\dagger_\p)^+=\Sel_{\rm Gr}(K,V^\dagger_\p)^{\Gal(K/\Q)},\quad\Sel_{\rm Gr}(K,V^\dagger_\p)^-=\Sel_{\rm Gr}(K,V^\dagger_\p\otimes\epsilon_K)^{\Gal(K/\Q)} \] 
(for the last equality use the fact that $\epsilon_K(\tau)=-1$). Restriction maps give isomorphisms 
\[ \res_K:H^1(\Q,V^\dagger_\p)\overset\simeq\longrightarrow H^1(K,V^\dagger_\p)^{\Gal(K/\Q)} \]
and
\[ \res_K:H^1(\Q,V^\dagger_\p\otimes\epsilon_K)\overset\simeq\longrightarrow H^1(K,V^\dagger_\p\otimes\epsilon_K)^{\Gal(K/\Q)},  \]
a consequence of the fact that the relevant kernels and cokernels are $2$-torsion groups as well as $F_\p$-vector spaces. To conclude, we have to show that the last two isomorphisms respect the local conditions defining Selmer groups. We deal with both cases simultaneously, letting $V$ denote either $V_\p^\dagger$ or $V_\p^\dagger\otimes\epsilon_K$. Thus we need to check that the isomorphism 
\[ \res_K:H^1(\Q,V)\overset\simeq\longrightarrow H^1(K,V)^{\Gal(K/\Q)} \] 
takes $\Sel_{\rm Gr}(\Q,V)$ isomorphically onto $\Sel_{\rm Gr}(K,V)^{\Gal(K/\Q)}$. Since, with notation as before, $D_v\subset D_\ell$ and $I_v\subset I_\ell$, the image of $\Sel_{\rm Gr}(\Q,V)$ under $\res_K$ is contained in $\Sel_{\rm Gr}(K,V)$, hence in $\Sel_{\rm Gr}(K,V)^{\Gal(K/\Q)}$. To prove surjectivity, fix $c\in \Sel_{\rm Gr}(K,V)^{\Gal(K/\Q)}$ and let $\tilde c\in H^1(K,V)$ 
be such that $\res_K(\tilde c)=c$: we want to show that $\tilde c\in\Sel_{\rm Gr}(\Q,V)$. In other words, our goal is to prove that $\res_\ell(\tilde c)\in H^1_{\rm Gr}(\Q_\ell,V)$ for every prime number $\ell$. First suppose that $\ell\neq p$. Since $I_\ell=I_v$ whenever $\ell$ is unramified in $K$, and since we already know that $\res_v(c)=0$ in $H^1(I_v,V)$, we can assume that $\ell$ is ramified in $K$, in which case $[I_\ell:I_v]=2$. But restriction from $I_\ell$ to $I_v$ yields an isomorphism 
\[ H^1(I_\ell,V)\overset\simeq\longrightarrow H^1(I_v,V)^{I_\ell/I_v} \] 
(kernel and cokernel are both $2$-torsion groups and $F_\p$-vector spaces) and $\res_v(c)$ actually belongs to $H^1(I_v,V)^{I_\ell/I_v}$, which shows that $\res_\ell(\tilde c)=0$. Suppose now that $v$ lies above $p$. Since $[D_p:D_v]\leq2$, there is an isomorphism 
\[ H^1(D_p,V^-)\overset\simeq\longrightarrow H^1(D_v,V^-)^{D_p/D_v} \] 
(again, kernel and cokernel are $2$-torsion $F_\p$-vector spaces) and $\pi(c)$ actually belongs to 
$H^1(D_v,V^-)^{D_p/D_v}$, which gives $\pi(\tilde c)=0$. Therefore $\tilde c\in\Sel_{\rm Gr}(\Q,V)$, as was to be shown. \end{proof}
 
\subsection{Twists of Hida families and low rank results}

The representations $V_\p\otimes\epsilon_K$, or equivalently the cusp forms $f_\p\otimes\epsilon_K$, can be $p$-adically interpolated by a Hida family $\gamma_\infty$ of tame level $ND^2$, in the same way as the $V_\p$ and the $f_\p$ are interpolated by $\phi_\infty$ (cf. \cite[Section 7]{DG}); such a Hida family will play an auxiliary role in our subsequent considerations. More precisely, if $\mathfrak H_\infty^\ord$ denotes Hida's ordinary Hecke algebra over $\cO_F$ of tame level $\Gamma_0(ND^2)$ then 
\[ \gamma_\infty:\mathfrak H_\infty^\ord\longrightarrow\mathcal R' \]
where $\mathcal R'$ is the branch of the Hida family on which $f\otimes\epsilon_K$ lives. In particular, the ring $\mathcal R'$ is a complete local noetherian domain, finite and flat over $\Lambda$, which is defined in terms of projective limits of Hecke algebras of tame level $ND^2$. As before, the Hida family $\gamma_\infty$ can be viewed as a formal power series $g_\infty\in\mathcal R'[\![q]\!]$, whose specializations at arithmetic primes $\mathfrak Q$ of $\mathcal R'$ will be denoted $g_{\mathfrak Q}$. It follows that if $\p$ is an arithmetic prime of $\mathcal R$ then $f_\p\otimes\epsilon_K=g_\fP$ for a suitable arithmetic prime $\fP$ of $\mathcal R'$ of weight $k_\p$. In the following, whenever we deal with an arithmetic prime $\p$ of $\mathcal R$ we shall use the symbol $\fP$ to indicate the arithmetic prime of $\mathcal R'$ such that $g_\fP=f_\p\otimes\epsilon_K$ (and analogously for $\p'$ and $\fP'$); more generally, we adopt lowercase gothic letters to denote primes of $\mathcal R$, while we reserve uppercase letters for primes of $\mathcal R'$.

One can attach to $\gamma_\infty$ a big Galois representation $\T'$ in exactly the same manner as $\T$ is associated with $\phi_\infty$ (see \cite[\S 5.5]{LV}), and there are specializations $W_{\mathfrak Q}:=\T'\otimes_{\mathcal R'}F'_{\mathfrak Q}$ of $\T'$ at arithmetic primes $\mathfrak Q$ of $\mathcal R'$, where $F'_{\mathfrak Q}$ is the residue field of $\mathcal R'$ at $\mathfrak Q$. In particular, with the convention introduced before, $W_\fP\simeq V_\p\otimes\epsilon_K$ for every arithmetic prime $\p$ of $\mathcal R$. Finally, write $\T'^{,\dagger}$ for the $\Theta^{-1}$-twist of $\T'$, so that $W_\fP^\dagger\simeq V_\p^\dagger\otimes\epsilon_K$ for every arithmetic prime $\p$ of $\mathcal R$, and introduce the notion of \emph{generic} weight $2$ arithmetic prime of $\mathcal R'$ in terms of an $\mathcal R'^{,\times}$-valued character $\theta'$ defined as in Section \ref{hida-sec}.    

\begin{theorem} \label{complementary-thm}
Assume that there exists an arithmetic prime $\p$ of even weight and trivial character such that 
\begin{itemize}
\item[(a)] $(f_\p,{\bf 1})$ has analytic rank zero;  
\item[(b)] neither $f_\p$ nor $f_\p\otimes\epsilon_K$ has an exceptional zero.
\end{itemize} 
Then
\begin{enumerate}
\item $\widetilde H^1_f(K,\T^\dagger)$ is $\mathcal R$-torsion;
\item $(f_{\p'},{\bf 1})$ has analytic rank zero for all but finitely many arithmetic primes $\p'$ of even weight and trivial character. 
\end{enumerate}
\end{theorem}

\begin{proof} First of all, for every arithmetic prime $\mathfrak q$ with trivial character there is a factorization of $L$-functions
\begin{equation} \label{L-splitting-eq}
L_K(f_{\mathfrak q},s)=L_K(f_{\mathfrak q},\boldsymbol1,s)=L(f_{\mathfrak q},s)\cdot L(f_{\mathfrak q}\otimes\epsilon_K,s). 
\end{equation}
Now let $\p$ be as in the statement of the theorem. Since $(f_\p,\boldsymbol1)$ has analytic rank zero, it follows from \eqref{L-splitting-eq} that $L(f_\p,k_\p/2)\neq0$ and $L(f_\p\otimes\epsilon_K,k_\p/2)\neq0$. Applying the results of Kato (\cite{Ka}) to the Hida families $\phi_\infty$ and $\gamma_\infty$ as in the proof of \cite[Theorem 7]{Ho2}, and recalling that $W_{\fP'}^\dagger\simeq V_{\p'}^\dagger\otimes\epsilon_K$, we obtain that 
\[ \widetilde H^1_f\bigl(\Q,V^\dagger_{\p'}\bigr)=0,\qquad\widetilde H^1_f\bigl(\Q,V^\dagger_{\p'}\otimes\epsilon_K\bigr)\simeq\widetilde H^1_f\bigl(\Q,W^\dagger_{\fP'}\bigr)=0 \]
for infinitely many arithmetic primes $\p'$ of $\mathcal R$. Now Proposition \ref{prop5.8} gives $\widetilde H_f^1(K,V^\dagger_{\p'})=0$ for infinitely many $\p'$, and for part (1) we conclude as in the proof of Theorem \ref{teo-def}. To prove part (2) we apply \cite[Theorem 7]{Ho2} to the families $\phi_\infty$ and $\gamma_\infty$ and get that $L(f_{\p'},k_{\p'}/2)\neq0$ and $L(f_{\p'}\otimes\epsilon_K,k_{\p'}/2)\neq0$ for all but finitely many $\p'$ of $\mathcal R$ of even weight and trivial character. On the other hand, equality \eqref{L-splitting-eq} with $\mathfrak q=\p'$ implies that
\[L_K(f_{\p'},k_{\p'}/2)=L(f_{\p'},k_{\p'}/2)\cdot L(f_{\p'}\otimes\epsilon_K,k_{\p'}/2), \] 
and we are done. \end{proof}  

Under narrower assumptions, we can also offer a result in rank one analogous to Theorem \ref{complementary-thm}. Here $w=\pm1$ is the sign appearing in the functional equation for the Mazur--Tate--Teitelbaum $p$-adic $L$-function (see \cite[Proposition 2.3.6]{Ho1}).  

\begin{theorem} \label{complementary-thm-bis}
Assume that $w=1$ and there exists an arithmetic prime $\p$ of weight $2$ and trivial character such that 
\begin{itemize}
\item[(a)] $(f_\p,{\bf 1})$ has analytic rank one;  
\item[(b)] $\p\in\Gen_2(\theta)$ and $\fP\in\Gen_2(\theta')$.
\end{itemize} 
Then $\widetilde H^1_f(K,\T^\dagger)$ is an $\mathcal R$-module of rank one.
\end{theorem}

\begin{proof} Let $\p$ be as in the statement. In light of factorization \eqref{L-splitting-eq}, two possibilities can occur:  
\begin{enumerate}
\item[(i)] $L(f_\p,1)\neq 0$, $L(f_\p\otimes\epsilon_K,1)=0$ and $L'(f_\p\otimes\epsilon_K,1)\neq0$.
\item[(ii)] $L(f_\p,1)=0$, $L(f_\p\otimes\epsilon_K,1)\neq0$ and $L'(f_\p,1)\neq0$.
\end{enumerate}
Suppose that we are in case (i). By \cite[Theorem 7]{Ho2} applied to the Hida family $\phi_\infty$, the form $f_{\p'}$ has analytic rank zero and $\widetilde H^1_f(\Q,V_{\p'}^\dagger)=0$ for all but finitely many $\p'\in\Gen_2(\theta)$. On the other hand, by \cite[Theorem 8]{Ho2} applied to the Hida family $\gamma_\infty$, the form $f_{\p'}\otimes\epsilon_K=g_{\fP'}$ has analytic rank one and $\widetilde H^1_f(\Q,V_{\p'}^\dagger\otimes\epsilon_K)\simeq\widetilde H^1_f(\Q,W_{\fP'}^\dagger)$ is one-dimensional over $F_{\p'}=F'_{\fP'}$ for all but finitely many $\fP'\in\Gen_2(\theta')$ (see the proof of \cite[Corollary 3.4.3]{Ho1} with $\Q$ in place of $K$). Finally, Proposition \ref{prop5.8} implies that $\widetilde H_f^1(K,V^\dagger_{\p'})$ is one-dimensional for almost all $\p'$, and the theorem follows by arguing as in the final part of the proof of \cite[Corollary 3.4.3]{Ho1}. Case (ii) is symmetric to (i) and is left to the reader. \end{proof}

\end{document}